\documentclass[a4paper,11pt,reqno]{amsart}
\usepackage[left=2.75cm,right=2.75cm]{geometry}
\usepackage{fouche/fouche}
\usepackage[sc]{mathpazo}
\NewDocumentCommand{\rat}{}{\mathrm{Rat}}

\NewDocumentCommand{\malg}{ O{} m }{%
	#1\langle\!\!\langle#2\rangle\!\!\rangle
}
\NewDocumentCommand{\yon}{}{\bsy}
\NewDocumentCommand{\concat}{}{%
	\mathchoice
	{\mathop{\scalebox{.75}{$\boldsymbol{+}\kern-.2em\boldsymbol{+}$}}}
	{\mathop{\scalebox{.55}{$\boldsymbol{+}\kern-.2em\boldsymbol{+}$}}}
	{\mathop{\scalebox{.45}{$\boldsymbol{+}\kern-.2em\boldsymbol{+}$}}}
	{\mathop{\scalebox{.25}{$\boldsymbol{+}\kern-.2em\boldsymbol{+}$}}}
}
\def\bang{\boldsymbol{!}}
\NewDocumentCommand{\emptyList}{}{\texttt{[\,]}}
\NewDocumentCommand{\kons}{}{\mathbin{::}}
\NewDocumentCommand{\sfRat}{ O{\clV} }{\mathbf{Rat}_{#1}}

\NewDocumentCommand{\presh}{ m O{\clV} }{\VCat[#2][{#1}^\op,#2]}
\NewDocumentCommand{\psh}{ m O{\clV} }{\VCat[#2][{#1}^\op,#2]}
\NewDocumentCommand{\tranComp}{}{\circ}

\NewDocumentCommand{\ratProf}{ O{\clB} }{\freerig{#1}\emdash\Prof}

\NewDocumentCommand{\TwoTDX}{}{\mathbf{2}\TDX}
\NewDocumentCommand{\tdto}{ O{} }{\mathrel{\xymatrix@C=6mm@1{*{} \ar[r]|(.4)@-{*}^{#1} & *{}}}}
\NewDocumentCommand{\pto}{ O{} }{\mathrel{\xymatrix@C=6mm@1{*{} \ar[r]|(.4)@-{|}^{#1} & *{}}}}
\newcommand{\laurent}[2]{#1(\!(#2)\!)}

\NewDocumentCommand{\tcomp}{ m m m O{x} O{y} O{b} O{\clB} }{\int^{\underline{#6} : {#7}^*} #1(\underline{#3}, #4,#4')[\underline{#6}]\times #2(\underline #6;#5,#5')}

\NewDocumentCommand{\toadd}{ m O{red} }{
	\todo[inline,caption=0,color=#2!10,bordercolor=white]{\color{#2!50!black}#1}
}
\NewDocumentCommand{\Yan}{ O{\clA} }{\Lan_{\yon_{#1}}}

\usepackage{xspace}
\def\ie{i.e.\@\xspace}
\def\eg{e.g.\@\xspace}
\def\cf{cf.\@\xspace}

\def\Id{\mathrm{Id}}

\NewDocumentCommand{\opid}{ m }{{#1}^\op\times #1}

\newcommand{\fugue}[4]{
	#1(#2,#3)\otimes #1(#2\ltimes #3, #4)
}

\newcommand*\circled[1]{\tikz[baseline=(char.base)]{
		\node[shape=circle,draw,inner sep=2pt] (char) {#1};}}

\def\tab{\text{tab}}%{\boldsymbol{\mathbb{T}}}
\def\cotab{\text{cotab}}%{\boldsymbol{\mathbb{K}}}

\def\pow#1{\text{P}({#1})}
\def\twoDash{2\hyp{}}
\usepackage{quiver}

\usetikzlibrary{decorations.pathmorphing}
\usetikzlibrary{calc,matrix,arrows,positioning,backgrounds,fit}
\usetikzlibrary {arrows.meta}
\pgfkeys{
  cm to/.tip={
    Classical TikZ Rightarrow[
      length = +1.05pt 11,
      width' = +0pt 0.65,
      line width = 0pt 0.9 1,
    ]
  },
  cm double to/.tip = {cm to[sep=+0pt -8.9]cm to},
}

\def\collage#1#2#3{#1\uplus_{#2}#3}

\NewDocumentCommand{\arar}{O{r} m m}{%
	\ar@<.3em>[#1]^-{#2}\ar@<-.3em>[#1]_-{#3}
}

\def\DTDX{\bbD\TDX}
\def\DProf{\bbP\mathbf{rof}}
\def\VMealy{\clV\Mealy}
\def\Bool{\bfB}
\NewDocumentCommand{\freerig}{ m O{} }{\bfK_{#2}\langle\!\!\langle{#1}\rangle\!\!\rangle} 
\title{Two-dimensional transducers}
\author{Fosco Loregian}
\address{Tallinn University of Technology}
\email{fosco.loregian@gmail.com}
\dedicatory{Dedicated to Bob Paré, on the occasion of his 80$^{th}$ birthday.}
\begin{document}
\begin{abstract}
	We define a bicategory \(\TwoTDX\) whose 1-cells provide a categorification of \emph{transducers}, computational devices extending finite-state automata with output capabilities. This bicategory is a mathematically interesting object: its objects are categories $\clA,\clB,\dots$ and its 1-cells \((\clQ, t) : \clA \tdto \clB\) consist of a category $\clQ$ of `states', and a profunctor% can be seen as `families of profunctors over \(\clQ\), indexed over \(\clA\), and enriched over (a universal monoidal category obtained from) \(\clB\)'; more precisely, 2-transducers are functors of type
\[\notag
	\vxy{
		t : \clA \times \opid{\clQ} \times (\clB^*)^\op \ar[r] & \Set
	}
\]
where \(\clB^*\) denotes the free monoidal category over \(\clB\). Extending \(t\) to \(\clA^*\) in a canonical way, to each `word' \(\una\) in \(\clA^*\) one attaches an endoprofunctor over the category \(\clQ\) of states, enriched over presheaves on \(\clB^*\).

We discuss a number of other characterizations of the hom-category \(\TwoTDX(\clA,\clB)\); we establish a Kleisli-like universal property for \(\TwoTDX(\clA,\clB)\) and explore the connection of \(\TwoTDX\) to other bicategories of computational models, such as Bob Walters' bicategory of `circuits'; it is convenient to regard \(\TwoTDX\) as the loose bicategory of a double category \(\DTDX\): the bicategory (resp., double category) of profunctors is naturally contained in the bicategory (resp., double category) \(\TwoTDX\) (resp., \(\DTDX\)); we study the completeness and cocompleteness properties of \(\DTDX\), the existence of companions and conjoints, and we sketch how monads, adjunctions, and other structures/properties that naturally arise from the definition work in $\DTDX$.
\end{abstract}

\maketitle
\tableofcontents

\section{Introduction}\label{introduction}
This paper analyzes the properties of a class of computational machines known as \emph{transducers} —an important concept in theoretical computer science, especially in the study of formal languages and automata theory— from a category-theoretic standpoint. Our perspective is that of a \emph{categorification}, building on the classical notion long studied by computer scientists. In this sense, the present work lives within a line of research (cf.\@\xspace \cite{Kasangian1990,kasangian1983cofibrations,Betti1982,Betti1981}, \cite{Guitart1980,guitart1978bimodules,guitart1974remarques}, and \cite{Katis2002}) that the author has been attempting to revive for the last few years (cf.\@\xspace\cite{EPTCS397.1,boccali_et_al:LIPIcs.CALCO.2023.20,loregian2024automata,loregian2025monadslimitsbicategoriescircuits}), advocating the use of \emph{formal category theory} as a foundational framework for automata theory.

This perspective reveals intriguing patterns, both for the category theorist — who may discover concrete instances of concepts traditionally viewed as highly abstract by other mathematicians — and for the automata theorist — who may gain access to a powerful, systematic framework for compositional properties of their familiar structures of interest.

Paré has contributed to this part of category theory with his elegant paper \cite{Par2010} on `Mealy morphisms' between enriched categories. There, he showed that by abstracting appropriately from the standard mathematical definition of a Mealy automaton, one arrives at a bicategory in which the 1-cells occupy a position intermediate between enriched functors and enriched profunctors. More indirectly, Paré's work embodies an aesthetic perspective that resonates strongly with the author, which is the idea that a mathematician should never shy away from asking seemingly naïve questions (or even `two silly ones', as in \cite{pare2021three}), since the answers may reveal unexpected analogies.

\medskip
Here, the `silly question' we seek to address is:
\begin{quote}
	How to convince a mathematician that a transducer is an interesting object to study, aided and motivated by category theory?
\end{quote}
In the pursuit of a (hopefully) satisfactory answer, the paper also engages with several other areas of category theory explored by Paré, including but not limited to the theory of profunctors, double categories \cite{grandispare1999limits}, categories enriched over presheaves \cite{wood1976indicial}, and graded monads, which he touched in \cite{Par2021}. In our account, a \twoDash transducer emerges as a 1-cell in a bicategory \(\TwoTDX\) (or better, a loose morphism in a double category \(\DTDX\)) of `graded-and-indexed' profunctors. The aim of the present work is to make this idea precise, and explore the properties of \(\TwoTDX\) and \(\DTDX\).\footnote{This notational convention will be implicitly assumed to distinguish categories and 2-categories (bold) from double categories (bbold and bold).}

\medskip
Having drawn constant inspiration and insight from Paré's work, it is with great pleasure and gratitude that we submit this work for consideration.
\subsection{Motivation, background and the idea behind transducers}
Let's now turn to the mathematics, starting with a brief introduction on the classical notion of transducer, and how it can be seen as a profunctor.

Shortly put, transducers are computational devices that extend finite-state automata with the ability to produce output; as such, they can be thought of as \emph{translation machines}, or converters, between different languages: at the bare essential, a finite word \(\una =a_1a_2\dots a_n\) (say, in an alphabet \(A\)) goes in, a computation \(t\) happens, and a set of finite words \(t(\una)\) (in a possibly different alphabet \(B\)) goes out. When such set is a singleton, \(t\) is called \emph{deterministic}.

So a transducer \(t\) can be thought of as a finite-state machine \emph{equipped with a dictionary}: for each transition \(x\xto\una x'\) between states induced by the word \(\una\), \(t\) produces a corresponding (set of) output(s). As such, transducers can serve a myriad of purposes, such as models for string-processing systems ---simple character encoders or complex signal processors-- and they find widespread application in many areas.

State-of-the-art research has already attempted to study transducers through a category-theoretic lens, for example using tools from the theory of coalgebras \cite{silva2010subsequential,adamek2011coalgebraic}. The theory lends itself to many generalizations and variations, such as \emph{weighted} \cite{mohri2001weighted} or \emph{probabilistic} \cite{eisner2002parameter} transducers.

Here we take a deeper category-theoretic stance, motivated by the fact that (despite their broad applicability and interest from the community of algebraists, see \eg \cite{Choffrut2004}), the existing literature does not seem to expand on the idea that, as dictionaries (and more importantly, databases, \cf \cite{spipacchione}) can be fruitfully thought as profunctors, so should be transducers: given alphabets $A,B$, a profunctor from \(A^*\) to \(B^*\) assigns to each word \((w,w') :  A^*\times B^*\) a set \(p(w,w')\), which can be thought of as the set of all ways one can effectively translate \(w\) into \(w'\) passing from \(A\) to \(B\).

The fact that transducers admit a compositional structure (translation can happen in series) and a monoidal structure (translation can, at least in principle, happen in parallel) is mentioned by many (\cf for example \cite[§3]{Mohri2004}) but isn't made into a precise assertion of compositionality/monoidality of a category of which transducers are arrows. 

In this work we take the above analogy `profunctors \(\approx\) dictionaries' seriously, and we aim to describe said compositional and monoidal structure. The answer we find unveils several analogies, requiring us to dive into delightfully sophisticated category theory.

\medskip
To attain this goal, the classical notion of (1-)transducer is first categorified, and transducers are seen as 1-cells of a bicategory (or more conveniently, as loose morphisms of a double category) \(\TwoTDX\) of \emph{\twoDash transducers}; then a bicategory of (1-)transducers can be seen as made by \twoDash transducers between discrete categories, and having discrete state categories. We then turn to explore various properties of \(\TwoTDX\), \emph{qua} bicategory and horizontal bicategory of a double category \(\DTDX\); in particular, we discuss several universal constructions, as well as monads and adjunctions in \(\DTDX\), and the relation of $\DTDX$ with other models for computational (`Mealy') machines.

\medskip
Now that the plan has been disclosed, let's start at the beginning. The idea of a (nondeterministic) transducer as a dictionary is mathematically captured defining it as a tuple \(T = (Q, A, B, t)\) where
\begin{enumtag}{td}
	\item \label{td_0} \(Q\) is a set of \emph{states} (often called a state `space' even when no topology whatsoever is involved),
	\item \label{td_1} \(A\) is the \emph{input} alphabet (on which we consider the set of words \(A^*\)),
	\item \label{td_2} \(B\) is the \emph{output} alphabet (on which we consider the set of words \(B^*\)),
	\item \label{td_3} \(t \subseteq A \times Q \times Q\times B^* \) is a \emph{transition relation}, specified so that there is a unique extension \(\bar t\) to a relation \(A^* \times Q \times Q\times B^* \).\footnote{More than often, \(Q,A,B\) are \emph{finite} sets; from a purely mathematical point of view, this restriction is a bit unnatural ---all the more because at least \emph{some} infinitary constructions are needed to make the theory expressive enough. There is a certain amount of work in re-enacting automata theory in minimal assumptions, such as `an elementary topos with a NNO', cf \cite{Iwaniack2024}. Grothendieck toposes offer a better framework, and it has been proposed that they should be the right setting to subsume `the technical details of automata theory', \cf \cite{hora2024topoiautomataitopoi,boccali2023semibicategory}. In short, this approach is motivated by the general fact that given a well-behaved enough endofunctor \(F\) on a topos, the \(F\)-coalgebras form themselves a topos.}
\end{enumtag}
As specified in \ref{td_3}, the transition relation \(t\) can be rewritten as a function of type
\[\label{generic_1_tdx}\vxy{
		A \times Q\times Q \ar[r] & \pow{B^*}
	}\]
	into the powerset of $B^*$, which in turn can be interpreted as a function assigning a \(Q\)-by-\(Q\) matrix of elements in the semiring \(\pow{B^*}\) to each element of \(A\), and by extension to each list \(w :  A^*\), so that the matrices \(t_a\) form a \emph{representation} of the free monoid \((A^*,\raisebox{2pt}{\(\concat\)},\emptyList)\) into the semiring \(\text{Mat}(Q,\pow{B^*})\) of such \(Q\)-by-\(Q\) matrices, \ie a monoid homomorphism
\[\label{generic_1_tdx_again}\vxy{t : (A^*,\raisebox{2pt}{\(\concat\)},\emptyList) \ar[r] & (\text{Mat}(Q,\pow{B^*}),\circ,\Id_Q).}\]
Let's consider a very small example, which will clarify the situation to those who meet the notion of transducer for the first time; let \(A = \{a,b\},B = \{x\}\)
denote the alphabets, \(Q=\{1,2\}\) the state space. A transducer  is specified in terms of two 2-by-2 matrices
\[\label{generic_1_tdx_matrc} t(a) = \begin{pmatrix}
		f_{11}(x) & f_{12}(x) \\
		f_{21}(x) & f_{22}(x)
	\end{pmatrix}
	\qquad\qquad t(b)=\begin{pmatrix}
		g_{11}(x) & g_{12}(x) \\
		g_{21}(x) & g_{22}(x)
	\end{pmatrix}
\]
whose entries \(f_{ij}(x),g_{ij}(x)\), dependent on a single indeterminate \(x\), are understood as characteristic functions \(\{x\}^*\to \{0,1\}\), belonging to the formal power series ring \(\malg[\Bool]{x}\) of (noncommutative, but when \(B\) is a singleton this is not evident) polynomials over the set \(B\) of indeterminates, and Boolean coefficients.\footnote{To each subset \(S\subseteq B^*\), denoted \(S=\{x^{n_s}\mid s :  S\}\), one associates the series \(\sum_{s :  S} x^{n_s}\); clearly, each Boolean-valued series of this form determines a subset of \(B^*\) as its support. So, the series \(0\) corresponds to the empty set \(\varnothing\subset B^*\), the series \(x\) corresponds to the subset \(\{[x]\}\) (where \([x] := x\kons\emptyList\)), the series \(x+x^2\) corresponds to \(\{[x],[x,x]\}\), etc.} Just for the sake of concreteness, let's say \(t(a)=\left(\begin{smallmatrix}
		1	&	0 \\
		0	&	x+x^2
	\end{smallmatrix}\right)
,t(b)=\left(\begin{smallmatrix}
		1	&	x \\
		x	&	x
	\end{smallmatrix}\right)\).

To such an arrangement one associates a directed graph representing the \emph{dynamics} of \(t\), where there is a node for each state \(q :  Q\) and an edge labeled \(q\xto{i/f} q'\), decorated by \((i,f) :  A\times \malg[\Bool]{x}\) if and only if \(t(i)_{qq'}=f\): for the \(t(a),t(b)\) above, we then have the diagram
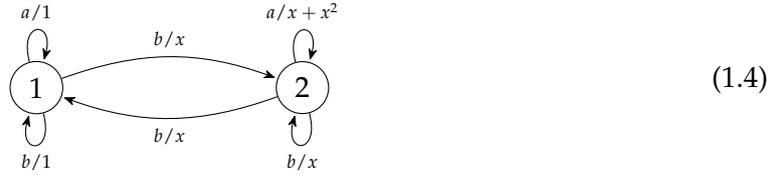
\[\label{gph_di_un_tdx}
	\begin{tikzpicture}[>={Stealth[round]}, shorten >=1pt, node distance=3.5cm, on grid, auto,baseline=(current bounding box.center)]
		\node[state,minimum size=0pt] (q1) {1};
		\node[state,minimum size=0pt] (q2) [right=of q1] {2};
		\path (q1) edge[loop above] node[font=\tiny] {$a/1$} ()
		edge[loop below] node[font=\tiny] {$b/1$} ();
		\path (q2) edge[loop above] node[font=\tiny] {$a/x+x^2$} ()
		edge[loop below] node[font=\tiny] {$b/x$} ();
		\path[->] (q1) edge[bend left=20] node[font=\tiny] {$b/x$} (q2)
		(q2) edge[bend left=20] node[font=\tiny] {$b/x$} (q1);
	\end{tikzpicture}
\]
Now, it's clear how to interpret \eqref{gph_di_un_tdx} in terms of an elementary notion: once a representation \(t\) of \(A^*\) as \eqref{generic_1_tdx_again} is regarded as a functor \(\Sigma(A^*) \to\Set\) out of the monoid $A^*$, regarded as a single-object free category \(\Sigma A^*\), one can turn \(t\) into an opfibration \(P_t : \clE[t]\to \Sigma (A^*)\), via the \emph{Grothendieck\hyp{}Bénabou} (or `category of elements') \emph{construction} \cite[§7]{Benabou2000}.

As such, \(\clE[t]\) is the category of elements, or \emph{action category} (by which we mean the non-groupoidal equivalent of action \emph{groupoids}, \cf \cite[Definition 4.7]{Ibort2019}) of \(t\), having objects the states \(q :  Q\) and morphisms the \(t(a)/f_{qq'} : q\to q'\) above. Given the equivalence of categories between functors \(\Sigma (A^*) \to\Set\) and discrete opfibrations over \(\Sigma (A^*)\), the action category faithfully represents the evolution in discrete time of the representation \(t\).

It is now an easy observation, first put in print by Walters \cite{Walters1989}, that if \(t\) is a representation of a free \emph{monoid}, then \(\clE[t]\) is a free \emph{category}: in particular, in our example,
\begin{remark}
	The action category \(\clE[t]\) associated with the copresheaf \(t\) of \eqref{generic_1_tdx_matrc} is free on the graph \eqref{gph_di_un_tdx}.
\end{remark}
This fibrational approach has a certain established history developed by Walters and extended by others, \cf \cite{Walters1989,mellies:hal-04399404,entics:10508}. However, this is just part of the story: every transducer \(t\) yields a \(2\times 2\) matrix for each word \(w : \{a,b\}^*\), uniquely determined as the product of the generating matrices \(t(a),t(b)\) in \eqref{generic_1_tdx_matrc}; for example, to the words `\(ab\)' and `\(ba\)' one associates the product matrices
\[t(ab)=t(a)t(b)=\begin{pmatrix}
		1       & x       \\
		x^2+x^3 & x^2+x^3
	\end{pmatrix}\qquad
	t(ba)=t(b)t(a)=\begin{pmatrix}
		1 & x^2+x^3 \\
		x & x^2+x^3
	\end{pmatrix}\]
and thus additional edges appear in the above diagram (of which we draw only some, shading in gray the generating edges):
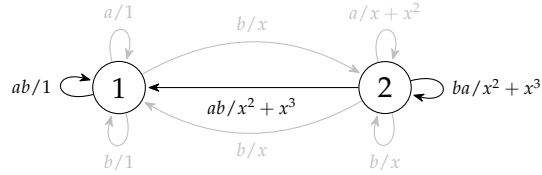
\begin{center}
	\begin{tikzpicture}[>={Stealth[round]}, shorten >=1pt, node distance=3.5cm, on grid, auto]
		\node[state,minimum size=0pt] (q1) {1};
		\node[state,minimum size=0pt] (q2) [right=of q1] {2};
		\path (q1) edge[gray!50,loop above] node[font=\tiny] {$a/1$} ()
		edge[gray!50,loop below] node[font=\tiny] {$b/1$} ()
		edge[loop left] node[font=\tiny] {$ab/1$} ();
		\path (q2) edge[gray!50,loop above] node[font=\tiny] {$a/x+x^2$} ()
		edge[gray!50,loop below] node[font=\tiny] {$b/x$} ()
		edge[loop right] node[font=\tiny] {$ba/x^2+x^3$} ();
		\path[->] (q1) edge[gray!50,bend left=30] node[font=\tiny] {$b/x$} (q2)
		(q2) edge[gray!50,bend left=30] node[font=\tiny] {$b/x$} (q1)
		(q2) edge node[font=\tiny] {$ab/x^2+x^3$} (q1);
	\end{tikzpicture}
\end{center}
Thus, in addition to being the total space of a discrete opfibration, in the above class of examples the action category \(\clE[t]\) is also \emph{graded}, in the sense that the composition of edges \(\circled{1}\xto{a/1}\circled{1}\xto{b/1}\circled{1}\) is the edge \(\circled{1}\xto{ab/1\cdot 1}\circled{1}\) arising as product of labels, and the same is true for the composition \(\circled{2}\xto{a/x+x^2}\circled{2}\xto{b/x}\circled{1}\), labeled by \(ab\big/(x+x^2)\cdot x\); more generally, given a map like \eqref{generic_1_tdx} and a word \(w=(a_1,\dots,a_n) :  A^*\), we express $\big(t(w)\big)_{ij}$ as a product
\[\label{profi_compozia}\big(t(w)\big)_{ij} = \big(t(a_1)\dots t(a_n)\big)_{ij} = \sum_{q_1,\dots,q_{n-1} :  Q} t(a_1)_{i,q_1}\cdot t(a_2)_{q_1,q_2} \cdots  t(a_n)_{q_{n-1},j},\]
where the sum and the iterated product is understood in terms of the semiring operations.

Categorification (colloquially known as `the left adjoint to the process of forgetting category theory') now provides us with an idea to make this story more appealing for a category theorist: a map like \eqref{generic_1_tdx}, decomposed as in \eqref{profi_compozia}, is a particular instance of a profunctor
\[\label{tru_trandu}\vxy{t : \clA\times\opid\clQ \ar[r] & \Set^{(\clB^*)^\op}}\]
when \(\clA,\clB,\clQ\) are discrete categories (regarded as categories enriched over \(\Bool=\{\varnothing,\top\}\)); in this perspective \eqref{profi_compozia} expresses the profunctor composition
\[\xymatrix{
	\clQ \ar[r]|-@{|}^{t(a_1)}& \clQ \ar[r]|-@{|}^{t(a_2)} & \dots\ar[r]|-@{|} & \clQ\ar[r]|-@{|}^{t(a_n)} & \clQ.
	}\]
So, a transducer can be formally described as a pair \((\clQ, t)\), consisting of a category and a profunctor as in \eqref{tru_trandu}; viewing $(\clQ,t)$ as a 1-cell \(\clA \tdto \clB\), any two such cells \((\clP,s) : \clA \tdto \clB\) and \((\clQ,t) : \clB \tdto \clC\), can be composed to form a 1-cell \((\clQ \times \clP, T\circ s) : \clA \tdto \clC\),\footnote{The profunctor \(T\circ s\) of type \(\clA\times\opid{(\clQ\times\clP)}\times(\clB^*)^\op \to\Set\) can be defined as a suitable composition of profunctors; what is interesting about this notion of composition, that we spell out in detail in \eqref{compozia}, is that a transducer \((\clQ,t) : \clA\tdto\clB\) can be regarded as a profunctor indexed over \(\clA\), graded over \(\clQ\), and enriched over \(\clB\).}  thus defining the bicategory \(\TwoTDX\), which is the subject of our interest. 

\medskip
Several notions are needed to make this precise, the most important of which is the \emph{grading} expressed by the fact that the category of states of a composite transducer is the product of the categories of states of its components.

We can introduce such grading into the picture in the following way (a perspective that we first learned from Paré's paper \cite{Par2021} on `morphisms of rings').
Recall what is known as a \emph{graded monad} (\cf \cite{fujii20192categorical} for a detailed, modern account; the definition however appears very early in the work of Bénabou):
\begin{definition}[Graded monad]
	Let \((\clM,\otimes,1)\) be a monoidal category. An \emph{\(\clM\)-graded monad} \((T,\eta,\mu)\) consists of a lax functor \(T : \clM \to\Cat\); as such, a graded monad comprises
	\begin{itemize}
		\item an endofunctor \(T_M : \clX\to\clX\) of a category \(\clX\), for every object \(M : \clM_0\);
		\item a graded unit (the \emph{unitor}) \ie a natural transformation of type
		      \[\vxy{\eta : \id_\clX \ar@{=>}[r] & T_1}\]
		      where \(1\) is the monoidal unit of \(\clM\);
		\item a graded multiplication (the \emph{compositor}), a family of natural transformations of type
		      \[\vxy{
				      \mu_{MM'} : T_M\circ T_{M'} \ar@{=>}[r] & T_{M\otimes M'},
			      }\]
	\end{itemize}
	subject to associativity and unitality axioms (see \cite[Definition 12]{Par2021}, or \cite[2.1.1]{fujii20192categorical}, for the explicit definition).
\end{definition}
Graded monads appear in our discussion in two (related) ways: first, there is a graded monad of matrices over any fixed commutative (semi)ring \(R\), a fact witnessed by how every \(p\times p\) matrix-of-matrices \(\bsM  :  \text{Mat}_p(\text{Mat}_q(R))\), each entry of which is a \(q\times q\) matrix \(M_{ij} :  \text{Mat}_q(R)\), can be flattened to a \(pq\times pq\) matrix. More precisely,
\begin{remark}[\protect{\cite[Proposition 13]{Par2021}}]\label{scalo_a_grado}
	There is a graded monad \(\text{Mat} : \bbN \to \Cat\) from the multiplicative monoid of natural numbers (regarded as a single-object, 2-discrete bicategory), into \(\Cat\), picking the category of rings and where each \(\text{Mat}_n\) acts sending a ring \(R\) to the ring of \(n\times n\) matrices with entries in \(R\). The monad unit consists of the isomorphism \(R\to\text{Mat}_1(R)\cong R\), and the multiplication is the function `flattening' a matrix of matrices \(\bsM\) into a matrix of scalars \(M\), having as first row the concatenation of first rows of all matrices in the first row of \(\bsM\), as second row the concatenation of all second rows of all matrices in the first row\dots, etc.: if \((p,q)=(2,2)\) for example, the multiplication acts as
	\[\label{flattanza}
		\begin{bmatrix}
			\left[
				\begin{smallmatrix}
					0 & 1 \\
					2 & 3 \\
				\end{smallmatrix}
			\right] & \left[
				\begin{smallmatrix}
					4 & 5 \\
					6 & 7 \\
				\end{smallmatrix}
			\right]          \\[.2em]
			\left[
				\begin{smallmatrix}
					8 & 9 \\
					10 & 11 \\
				\end{smallmatrix}
			\right] & \left[
				\begin{smallmatrix}
					12 & 13 \\
					14 & 15 \\
				\end{smallmatrix}
			\right]          \\
		\end{bmatrix}
		\mapsto
		\left[\begin{smallmatrix}
				0 & 1 & 4 & 5 \\
				2 & 3 & 6 & 7 \\
				8 & 9 & 12 & 13 \\
				10 & 11 & 14 & 15
			\end{smallmatrix}\right]\]
\end{remark}
In a similar fashion, one can resort to a categorified analogue of \autoref{scalo_a_grado} to compose two \twoDash transducers by multiplying the grading categories, meaning that the composition of \((\clP,s) : \clA \tdto\clB\) with \((\clQ,t) : \clB\tdto\clC\) is of the form \((\clQ\times\clP,\blank) : \clA \tdto\clC\).

But in addition to this, graded monads also appear when one considers monads in the bicategory \(\TwoTDX\), or better, in the double category \(\DTDX\). For instance, in \autoref{monad_in_TwoTDX} we prove the following result.
\begin{theorem*}
	A monad \((\clQ,t) : \clA\tdto\clA\) in \(\TwoTDX\) consists of a promonad on \(\clA^*\), such that \(\clQ\) acquires a monoidal structure \(\boxtimes\), and \(t\) is graded over \(\opid{\clQ}\) (equipped with the component-wise monoidal structure \((A,B)\boxtimes(C,D):=(A\boxtimes^\op C,B\boxtimes D)\) induced by \(\clQ\)).
\end{theorem*}
We now give a more detailed account of the results contained in the present work.
\subsection{Outline of the paper}
In order to prepare the ground for a categorification of the notion of transducer, we start \autoref{bic_2_transducers} by setting all the classical pieces and definitions in place; this will motivate the need to regard a map like \eqref{generic_1_tdx}
as a functor of type
\[\vxy{t : \clA \times\opid\clQ \ar[r] & \freerig\clB}\]
where \(\freerig\clB\) is a shorthand for the presheaf category over \(\clB^*\);
starting from here, we define the bicategory \(\TwoTDX\). This is a kind of structure drawing from diverse parts of category theory, such as classical enriched profunctors \cite[6.2.10]{Bor2}, and more modern attempts to categorify free semirings,\footnote{A rather well-established branch of monoidal category theory \cite{Kelly1980,laplaza1972coherence,elgueta2020groupoid,baez2023schurfunctorscategorifiedplethysm,baez20242rigextensionssplittingprinciple} that the author touched upon in \cite{Loregian2023,loregian2024automata}.} and with the theory of graded monads \cite{McDermott2022,10.1007/978-3-031-16912-0_4,Gaboardi2021,Orchard2020,fujii20192categorical,Dorsch2018GradedMA,Fujii2016,milius_et_al:LIPIcs:2015:5538}.%Transducers as 1-cells \(\clA\tdto\clB\) are best understood, from the category-theoretic point of view, as indexed-enriched profunctors, where \(\clA\) provides the indexing, and \(\clB\) the enrichment.

After some preliminaries, in \autoref{1_transduc} we recall what kind of 1-dimensional structure is a transducer: a certain monoid homomorphism \(t : A^* \to \text{Mat}(Q, \pow{B^*})\) for suitable alphabets \(A, B\) and a finite state space \(Q\), where \(\pow{B^*}\) denotes the free quantale on \(B^*\) (see \autoref{1_transduc}). These gadgets are naturally understood as \emph{matrices of formal power series}, and we leverage on this idea in \autoref{definition_2_transducer}, to define a \emph{\twoDash transducer} as a strong monoidal functor \(t : \clA^* \to \freerig\clB\emdash\Prof(\clQ, \clQ)\), where \(\freerig\clB\) is the free cocompletion \cite{DAY2007651} of the free monoidal category \(\clB^*\) on \(\clB\) (this is called the \emph{free cocomplete 2-rig} on \(\clB\), \cf \cite{Loregian2023}). Such $t$ form the 1-cells of a bicategory $\TwoTDX$.

Furthermore, when $\clA,\clB,\dots$ are discrete categories, one can think of them as the sub-2-category $\TDX$ spanned by the 2-transducers where the image of $t : A\times\opid Q\times (B^*)^\op\to\Set$ factors through the subcategory \(\{\varnothing,\top\} =: \Bool\) of truth values; in this case a \twoDash transducer boils down exactly to a transducer as classically defined --a graded version of profunctor composition defines composition of 1-cells for both \(\TDX\) and \(\TwoTDX\). To the best of our knowledge, the fact that 1-transducers are composable, although obvious, this doesn't seem to have been explicitly spelled out so far. 

\medskip 
The progression in categorification  similar to what happens with relations, regarded as special profunctors, and profunctors, regarded as the loose bicategory of a (pseudo) double category:
\[\vxy[@R=4mm@C=2cm]{
	\Rel \ar@{}[d]|-{\rotatebox{90}{\(\supseteq\)}}&\Prof \ar@{}[d]|-{\rotatebox{90}{\(\supseteq\)}}& \DProf \ar@{}[d]|-{\rotatebox{90}{\(\supseteq\)}} \\
	\TDX \ar@{^{(}->}[r]_-{\sm[x]{\text{spanned by}\\ \text{discrete objects}}}& \TwoTDX \ar@{^{(}->}[r]_-{\sm[x]{\text{spanned by}\\ \text{tightly discrete cells}}}& \DTDX.
	}\]
Moreover, 2-cells in the bicategory \(\TwoTDX\) are \emph{transducer transformations} —natural transformations between profunctors, respecting the grading via reindexing functors (\autoref{definition_the_bicategory_bitwotran}). The coherence conditions imposed by these 2-cells are studied in detail, we show how they become simple consequence of a double category of which \(\TwoTDX\) is the bicategory of tightly discrete cells; see \autoref{2cell_theta_triangle}, \autoref{whiskering_2cells}, \autoref{dbcat_of_tdx}. The entirety of \autoref{properties} is devoted to the study of universal constructions in the double category $\DTDX$ (limits and colimits, and the existence of companions and conjoints).

\medskip
We show in \autoref{remark_many_versions_of_a_category} that the hom-category \(\TwoTDX(\clA,\clB)\) admits several equivalent characterizations, given by the multiple ways in which a map of type
\[\vxy{
		\clA\times\opid\clQ\times(\clB^*)^\op \ar[r] & \Set
	}\]
can be curried and uncurried (\autoref{tdx_char_1}--\autoref{tdx_char_6} of \autoref{remark_many_versions_of_a_category}).

The least straightforward of such characterizations is expanded in \autoref{locally_graded_coKleisli}, finding that \(\TwoTDX(\clA,\clB)\) is a `graded coKleisli category' for a (graded) monad \(K_\clQ=(\blank\times\opid\clQ)\). These perspectives are all useful to concisely derive various (local, \ie hom-category-wise, and global) properties of \(\TwoTDX\); \cf \autoref{tdx_properties_and_corollaries} and \eqref{liftaggio_at_fibs} observing that essentially by definition each hom-category \(\TwoTDX(\clA,\clB)\) is cocomplete, and that pre- and post-composition preserve colimits (thus, right liftings and right extensions exist), and that \(\TwoTDX\) is fibred over \(\Cat\), inheriting an enrichment over \(\Psd(\Cat^\op,\CAT)\) (one can then conjecture that \(\TwoTDX\) is the truncation of a 3-category; the obstruction in this sense is spelled out in \autoref{nota_2_fun}). The characterization as graded coKleisli category of \(K_\clQ\) allows to determine that adjoint pairs in \(\TwoTDX\) are poorly behaved: they can only arise from ordinary adjunctions in \(\Cat\), with trivial state categories (\autoref{adjunction_in_coKl} and \autoref{very_few}).

\medskip
As it is common for other bicategories where 1-cells are profunctors, hom-categories \(\TwoTDX(\clA,\clB)\) between finite objects have a rather direct combinatorial interpretation. Let \(\clI\) be the terminal category; in \autoref{linear_proalgebra} we exhibit how the category \(\TwoTDX(\clI,\varnothing)\) behaves as a category of pairs (category, endoprofunctor on that category). This is in analogy with linear algebra, where a pair \((V,f)\) of a vector space and a linear map over said vector space is just the data of a \(k[X]\)-module structure on \(V\). In \autoref{transducers_low_cardinality}, we characterize transducers of type \(\varnothing\tdto\clI\), \(\clI\tdto\clI\), \(\varnothing\tdto\varnothing\) and the way they naturally act on \(\clL\); this has connections with the category of \(\bbN\)-graded sets, and with the construction of free promonads on an endoprofunctor, see \eqref{nbdy_expect_free_promonad}.

To add to the compositional structure of \(\TwoTDX\), we describe a monoidal structure on the locally posetal bicategory \(\TDX\) of 1-transducers, turning \(\TDX\) into a (strict) monoidal 2-category (\autoref{sec_monoidality}); such 2-monoidal structure heavily relies on the existence of a tensorial strength for the presheaf construction, \cf \autoref{stren}. If the composition of 1-cells accounts for the possibility to sequentially stack `graded processes', the monoidal structure accounts for a way to stack them in parallel, compatibly with (\ie functorially with respect to) the former operation. Something similar happens for other 2-categories of `processes', \cite{Katis2010,Katis1997,EPTCS397.1,boccali_et_al:LIPIcs.CALCO.2023.20,boccali2023semibicategory,openTransitionSystems21}, which ought to be monoidal bicategories.

\medskip
The theory of monads in \(\TwoTDX\) is touched upon in \autoref{whatsa_monad}. We show that a monad in \(\TwoTDX\), \ie a functor
\[\vxy{
		t : \clA\times\opid\clQ \times (\clA^*)^\op \ar[r] & \Set
	}\]
subject to monad axioms, carries a fairly rich structure: it is a \emph{monoidal promonad} on \(\clA^*\) and \emph{graded} over its state category \(\clQ\), which carries a monoidal structure induced by the multiplication and unit cells (\autoref{monad_in_TwoTDX}). The appearance of a grading is a known phenomenon when studying the formal theory of monads for such bicategories of processes; our recent work \cite{loregian2025monadslimitsbicategoriescircuits} outlines that a monad in a pseudo double category where 1-cells are Mealy automata consists of a process where, on top of the alphabet acting on states, states act on the alphabet, and these two actions `cohere' in a suitable sense.

This fact bears an evident resemblance to the fact that (\cf \cite{10.1145/3022670.2951939}) a distributive law between monads \(T : M\to\Cat,S : N\to\Cat\) graded by monoids \(M,N\) respectively boils down to a matching structure for \(M,N\) (\cf \cite{Agore_2014,Agore2009forfinite,Agore_2009,Agore_2010}), and the resulting composite monad \(ST\) is graded over the `bicrossed' (or `Zappa-Sz\'ep') product \(M\bowtie N\) of \(M,N\).

In \autoref{sec:more-general} we show how \(\TwoTDX\) relates with other bicategories of processes; we show how \(\TwoTDX\) equips the bicategory \(\KSW\) of \cite{Katis1997} with proarrows \cite{rosebrugh1988proarrows,wood1985proarrows,wood1982abstract}, in the sense that there is a locally fully faithful pseudofunctor \(\Pi : \Cat \to \TwoTDX\) so that every 1-cell in the image of \(\Pi\) has a companion and a conjoint in \(\TwoTDX\); interestingly, \(\Pi\) splits as a composition of two locally fully faithful pseudofunctors, neither of which is a proarrow equipment. This is the content of \autoref{equippy}. We also draw a connection with Guitart's category \(\MAC\) of machines described in \cite{guitart1974remarques} (but see \cite[§2]{EPTCS397.1} for a more modern cut on the construction, as well as for the relations between \(\Mly\) and \(\MAC\)).

\medskip
The paper ends with a speculation, in \autoref{future}, connecting profunctors-as-matrices with the classical theory of differential equations with regular singularities \cite{Wasow1965, Levelt1971,Boalch2001, Sabbah2007,Deligne1970, Malgrange1991} and differential algebra at large; in simple terms, there are settings where it is common to study systems of differential equations of type
\[Y'(z) = A(z)Y(z),\label{deligname}\]
where \(Y(z)\) is an \(n\)-tuple of indeterminate functions, and \(A(z)\) is a \(n\times n\) matrix of meromorphic functions (so, series in the field of fractions of formal power series), in the complex variable \(z\).

If one is willing to abandon some motivation coming from the theory of state machines, equation \eqref{deligname} can be categorified as follows: let \(\clQ=\Bij\) be the category of finite sets and bijections (so the free \emph{symmetric} monoidal category \(\clI^\sigma\) on a single generator), then one can consider
the category \(\clV_\clB:=\freerig\clB[\sigma]\) as base of enrichment, and the category of \emph{\(\clV_\clB\)-enriched combinatorial species} as the functor category with objects \(Y : \clQ^\op\to\freerig\clB[\sigma]\).

Fix any such functor \(Y : \clQ^\op\to\freerig\clB[\sigma]\), a profunctor \(A : \opid\clQ \to \freerig\clB[\sigma]\), and an equation like
\[\label{diffeq_delignosa_in_spc}\partial Y \cong A\otimes Y\]
can be interpreted as an isomorphism between the derivative of \(Y\) at the LHS of the equation, in the sense of the category of Joyal species,\footnote{The derivative of a species is defined as \(\partial Y[n]  = Y[n+1]\), \cf \cite[§1.4]{bergeron1998combinatorial}, \cf also \cite{rajan1993derivatives} for a neat elementary introduction and its categorical properties.} and the `matrix product'
\[A\otimes Y =\int^{n:\Bij} A(n,\blank)\otimes Yn : \clQ^\op\to\freerig\clB[\sigma]\]
at the RHS, now that \(A\) can be regarded as a \emph{symmetric \twoDash transducer} \(\clI\tdto[\sigma]\clB\), \ie a functor
\[\vxy{
		\clI^{\sigma}\times\opid\clQ \ar[r] & \freerig\clB[\sigma],
	}\]
strong monoidal in the first argument (because together with an \(A\) as in \eqref{deligname}, we access all its iterates).

From a distance, this resembles the notion of monoidal distributor \cite[2.1]{Gambino2017}, of which we would like to think the theory of \twoDash transducers as some sort of noncommutative, graded analogue. For the moment, however, we leave this question open for future work.
\section{The bicategory of \twoDash transducers}\label{bic_2_transducers}
We begin by recalling the classical construction of transducers, based on \cite[III.6]{Berstel1979} and \cite[2.3]{Choffrut2004}, with adaptations to clarify the categorification process introduced in \autoref{definition_the_bicategory_bitwotran}. Additional comprehensive references for the algebraic approach to transducers include \cite{book91519310,Kuich1985yp,Salomaa1978-sx}.

\medskip
Let \(A\) be a set (an `input alphabet'), and a base semiring \(\bbK\).\footnote{Often in the following \(\bbK\) is the set of Booleans, equipped with logical XOR and AND operations.} Denote as
\[A^* := \sum_{n\ge 0}A^n = 1 + A + A\times A + \dots + A^n + \dots\]
the countable coproduct of all \(n\)-fold products \(A^n\), for \(n\ge 0\). Elements of \(A^*\) are finite ordered lists. The set \(A^*\) is equipped with a monoid structure given by concatenation of lists, denoted as an infix symbol \(\raisebox{2pt}{\(\concat\)}\). This gives \(A^*\) the universal property of the \emph{free monoid} on the \emph{carrier set} \(A\).

The construction \(A\mapsto A^*\) is a functor \(\Set\to\Mon\) into the category of monoids, left adjoint to the functor \(\Mon\to\Set\) that forgets the monoid operation.
\begin{definition}[Monoid algebra of \(A\)]
	\label{definition_monoid_algebra_of_a}
	The \emph{monoid algebra} of \(A\) (with coefficients in \(\bbK\)) is the set, denoted \(\malg{A}\), of all functions \(f : A^*\to\bbK\) equipped with pointwise sum in \(\bbK\) as addition, and \emph{Cauchy convolution}
	\[(f \odot g)(l) = \sum_{p \concat q = l} f(p) \cdot g(q)\]
	as multiplication. As such, \((\malg{A},+,\odot)\) is itself a semiring with additive unit the constant at 0, and multiplicative unit the `Dirac delta' of the empty list, \(\delta([]) = 1\) and \(\delta(a\kons as)=0\).

	Note that \(\malg{A}\) has the universal property of the \emph{free noncommutative \(\bbK\)-algebra} on the set \(A\).
\end{definition}
On \(\malg{A}\) there's also a \emph{restricted star} operation, \ie a partial function \((\firstblank)^\star : f\mapsto f^\star := \sum_{n\ge0} f^{\odot n}\), where \(f^{\odot n} := f \odot \dots \odot f\), \(n\) times,
defined only on \emph{reduced} elements (the \(f\)'s such that \(f\emptyList = 0\)).\footnote{The restriction on the domain of \(\star\) is obviously due to the fact that the star of \(f(X)=a_0+Xg(X)\) is \(\Big(\sum_{k\ge 0}a_0^k\Big) + Xh(X)\), and thus, if \(a_0\ne 0\), additional conditions of boundedness or a topology on \(\bbK\) might be necessary to ensure the existence of \(f^\star\). Take two examples: the star of \(f(X)=\frac 1{1-X}\) (real coefficients) is the Laurent series \(\frac{X-1}X\), and the star of \(f(X)=1\) (again real coefficients) doesn't exist.}
\begin{definition}[Rational series on \(A\)]
	\label{definition_rational_series_on_a}
	The family \(\rat_\bbK(A^*)\) of \emph{rational series} over the semiring \(\bbK\) is the
	smallest subset \(S\) of \(\malg{A}\) such that
	\begin{itemize}
		\item the constant series 0 belongs to \(S\);
		\item all `monomial' terms \(k(a \kons \emptyList)\) for \((a,k)  :  A\times\bbK\) are in \(S\);
		\item if \(f,g :  S\) then \(f+g, f\odot g\) are in \(S\);
		\item if \(f :  S\) is reduced, then \(f^\star :  S\).
	\end{itemize}
\end{definition}
\begin{definition}
	\label{boolean_rational_series}
	If \(\Bool=\{0,1\}\) is the set of Booleans, given a set \(X\), \(\rat_\Bool(X)\) is the subset of rational Boolean-valued series inside the ring \(\malg[\Bool]{X}\).
\end{definition}
\begin{remark}
	\label{boolean_series_powerset}
	Note that \(\malg[\Bool]{X}\) is naturally identified with the powerset of \(X^*\), and thus characteristic functions \(\chi_U : X^* \to\Bool\) of subsets \(U\subseteq X^*\) are elements of \(\malg[\Bool]{X}\) in a natural way.
\end{remark}
\begin{remark}\label{free_quant_remark}
	The set \(\malg[\Bool]{X}\) exhibits the universal property of the \emph{free quantale} over the set \(X\). In fact, \(X^*\) is the free monoidal category over the discrete category \(X\), and \(\malg[\Bool]{X}\) is just the presheaf category over it, cocomplete (since it is a lattice) and symmetric monoidal closed (since it is equipped with pointwise concatenation of subsets, \ie Day convolution of \(\Bool\)-enriched presheaves).
\end{remark}
\begin{definition}
	\label{rational_relations}
	Given two sets \(A,B\) one defines the \emph{rational relations} \(R\subseteq A^*\times B^*\) as the smallest subset of \(\Bool^{A^*\times B^*}\) containing the singletons and closed under the operations of (countable) subset union, Day convolution, defined for relations \(R,S : A^*\times B^* \to \Bool\) as
	\[(R\odot S)(\unx,\uny) =\bigvee_{\substack{\una\concat \una'=\unx\\\unb\concat \unb' = \uny}} (\una \, R \, \unb)\land (\una'\, S \, \unb') \]
	and restricted star, defined as \(R^\star := \bigvee_{n\ge 0} R^{\odot n}\); the subset \(R\subseteq A^*\times B^*\) is a rational relation if and only if for every \(u :  A^*\) the series \(R_u : B^* \to \Bool\) defined as \(v\mapsto [\text{is }(u,v)\in R?]\) is rational as an element of \(\malg[\Bool]{B}\) in the sense of \autoref{definition_rational_series_on_a}.
\end{definition}
\begin{definition}[1-transducer]\label{1_transduc}
	A (finite) \emph{transducer} consists of a (finite) set \(Q\) and a monoid homomorphism \(A^* \to M(Q,\rat(B))\) into the (semi)ring of \(|Q|\times |Q|\) matrices with coefficients in \(\rat(B)\).
\end{definition}
So far, God's in his heaven, all is right with the world. Now, let's take all this up a notch, and consider a complete, cocomplete, symmetric monoidal closed base of enrichment \(\clV\).\footnote{In specific examples, the base \(\clV\) will be the category of sets, or the category of Booleans --regarded as the full subcategory of the former spanned by the objects \(\{\varnothing,\top\}\). The constructions we outline are valid whenever \(\clV\) is a Bénabou cosmos, such that \emph{both} tabulators and cotabulators in the double category of \(\clV\)-profunctors.} Given a \(\clV\)-category \(\clB\), the \(\clV\)-category \(\presh{(\clB^*)}\) is the usual one defined \eg in \cite{kelly}, having objects the \(\clV\)-presheaves \((\clB^*)^\op \to\clV\).
\subsection{Two-dimensional transducers}\label{categorified_tdxs}
It is straightforward to give the following definition of a monoidal category \(\sfRat(\clB)\), as a categorification of \(\rat(B)\) in \autoref{definition_rational_presheaves}. We record the definition for completeness, but we will not make much use of it, since there exist bases of enrichment that are far more natural from the category-theoretic standpoint.
\begin{definition}[Rational presheaf]
	\label{definition_rational_presheaves}
	Let \(\clA,\clB\) be two \(\clV\)-enriched categories; define \(\sfRat(\clB)\) as the smallest full subcategory of \(\presh{\clB^*}\) such that
	\begin{enumtag}{rp}
		\item \label{rp_1} the initial object \(\varnothing\) of \(\presh{\clB^*}\) is in \(\sfRat(\clB^*)\);
		\item \label{rp_2} all representables are in \(\sfRat(\clB^*)\);
		\item \label{rp_3} if \(F,G :  \sfRat(\clB^*)\), then the coproduct \(F+G\) and the Day convolution \(F\odot G\) are in \(\sfRat(\clB^*)\);
		\item \label{rp_4} if \(F : \sfRat(\clB^*)\) is reduced (\ie if \(F(\emptyList) = \varnothing\)), then \(F^\star\), defined as \(\sum_{n\ge 0}F^{\odot n}\), is also in \(\sfRat(\clB^*)\).
	\end{enumtag}
\end{definition}
Note that, as a consequence of \ref{rp_3}, \(\sfRat(\clB^*)\) consists of the closure under restricted star of the finite coproduct completion of \(\clB^*\). Such finite coproduct completion has the universal property of the \emph{free} (cocartesian) \emph{2-rig} on \(\clB\), in the sense of \cite{Loregian2023}. So, there is a certain universal property that \(\sfRat(\clB^*)\) satisfies.

\medskip
From this, one could define a \emph{\twoDash transducer} as a strong monoidal functor
\[\vxy{
		t : \clA^* \ar[r] & \sfRat(\clB^*)\emdash\Prof(\clQ,\clQ).
	}\]
It is however difficult to work with enriched categories and profunctors over \(\sfRat(\clB^*)\), as this base of enrichment is not cocomplete (the free cocompletion of a category \(\clC\) under a class of colimits \(\Psi\) can in principle acquire more colimits, but just those in the \emph{closure} of \(\Psi\) under iterating the formal addition of \(\Psi\)-shaped colimits; see \cite{Albert1988}). Although frameworks for addressing this issue exist, treating \(\sfRat(\clB^*)\)-enriched profunctors as a \emph{virtual double category}, they would lead us astray from this first step in understanding categorified transducers. So, motivated by the fact that in \autoref{1_transduc} above we have considered a posetal base of enrichment, where the free coproduct- and the free colimit-completion coincide, we instead refer to the more general, and more well-behaved, category of profunctors enriched over the free completion \(\presh{(\clB^*)}\) of \(\clB\). The category obtained from \(\clB\) first taking the free monoidal \(\clV\)-category \(\clB^*\) on \(\clB\), and considering the closure under coproducts of the image of the Yoneda embedding enjoys the universal property of the \emph{free} (cocomplete) \emph{2-rig} on \(\clB\), in the sense of \cite{Loregian2023}. This satisfies a stronger (better, for our purposes) universal property than \(\sfRat(\clB^*)\).

There is a vast literature on categories enriched over \(\clV\)-presheaves, dating back to R. Wood's PhD thesis \cite{Wood1978,wood1976indicial} (written under the supervision of Paré) and recently further studied and adapted in \cite{lucyshynwright2025vgradedcategoriesvwbigradedcategories,McDermott2022}. Applications of such categories to automata theory are not new, but only appear in a somewhat cryptic remark in \cite[2.1.5°]{Guitart1980} citing Wood's thesis as a reference for what is there called a `bicatégorie graduée'.

In the following, for the sake of concise notation, we write \(\freerig\clB[\clV]\) as shorthand for the category \(\presh{(\clB^*)}\), and \(\freerig\clB\) when \(\clV=\Set\), an assumption we make now until further notice. This is meant to evoke the idea that \(\freerig\clB\) contains `polynomials' in variables determined by the category \(\clB\).

\medskip
This long series of preliminary observations finally paves the way to the following
\begin{definition}[\twoDash transducer]
	\label{definition_2_transducer}
	A \emph{\twoDash transducer} \(t : \clA\tdto\clB\) consists of a pair \((\clQ,t)\) where \(\clQ\) is a category and \(t\) is a strong monoidal functor
	\[\label{transduc_eqn}\vxy{
			t : \clA^* \ar[r] & \ratProf(\clQ,\clQ),
		}\]
	where \(\ratProf(\clQ,\clQ)\) has objects the functors \(\opid{\clQ} \to \freerig{\clB}\) (or, which is equivalent, \(\freerig{\clB}\)-enriched functors with domain the free \(\freerig{\clB}\)-enriched category over \(\opid{\clQ}\)), and \(\freerig{\clB}\)-enriched natural transformations, and monoidal structure given by composition of 1-cells.
\end{definition}
Giving evidence that this definition makes sense is a process composed of two parts:
\begin{itemize}
	\item the definition `unwinds well' and actually reduces to the old one when instantiated in the more canonical setting; ultimately, this follows from \autoref{free_quant_remark}.
	\item The definition sheds a better light on the decategorified topic.
\end{itemize}
Establishing this point is the purpose of this note. To do so, let's unpack the type of \(t\) in \eqref{transduc_eqn}: it uncurries to a functor \(\xymatrix{t : \clA^* \times \opid{\clQ} \ar[r] & \freerig{\clB}}\)
such that there are structural isomorphisms
\[\notag t(l\mathbin{\raisebox{2pt}{\(\concat\)}} l',\firstblank,\firstblank)\cong t(l,\firstblank,\firstblank)\diamond t(l',\firstblank,\firstblank)\quad \text{and} \quad t(\emptyList,\firstblank,\firstblank) \cong\hom_\clQ(\firstblank,\firstblank)\]
compatibly with the bicategory structure (associators and unitors). In turn, such a \(t\) further uncurries into a functor \(\xymatrix{t : \clA^* \times \opid{\clQ} \times (\clB^*)^\op \ar[r] & \Set}\)
qualifying each \(t_a\) as a `\(\clQ\)-by-\(\clQ\) matrix of \(\clB\)-valued power series'.

Given that the domain category of a \twoDash transducer \(t\) is free monoidal, \(t\) is uniquely determined by its restriction to \(\clA\), \ie on `generators' \((a\kons \emptyList)\) of \(\clA^*\), so that in fact \(t\) is completely determined by a functor \(t_0\) of type \(\clA \times \opid{\clQ} \to \freerig{\clB}\). In the following we will blur the notational distinction \(t\) and \(t_0\) as this is rarely a source of confusion, but it is worth to stress again the obvious fact that a transducer is \emph{either} a mere functor with domain \(\clA\), \emph{or} a strong monoidal one with domain \(\clA^*\).
\begin{remark}[Structure constants of a \twoDash transducer]
	\label{remark_structure_constants_of_a_2_transducer}
	We term the profunctors \(t_a = t_0(a\kons \emptyList)  :  \freerig{\clB}\emdash\Prof(\clQ,\clQ)\) the \emph{structure constants} of the \twoDash transducer, in analogy with representation theory \cite{humphreys1972introduction,serre1992lie}: given a finite-dimensional Lie algebra \(\fkg\) (with a basis \(\{e_1,\dots,e_n\}\) for its underlying vector spaces), the scalars \({\gamma_{ab}}^c\) appear as coordinates of the Lie bracket of two basis element \([e_a,e_b] = \sum_c {\gamma_{ab}}^c \, e_c\).
\end{remark}
\begin{remark}\label{strong_monoidality_2transducer}
	It is worth to spell out in detail what strong monoidality means for a \twoDash transducer; for any list of objects \(\una=(a_1,\dots,a_n) : \clA^*\), the functor \(t\) is determined on the generators in the sense that
	\[t(\una,\firstblank) : (q,q',\unb)\mapsto \int^{x_1,\dots,x_{n-1}}t_{a_1}(q,x_1)(\unb)\times t_{a_2}(x_1,x_2)(\unb)\times\dots\times t_{a_n}(x_{n-1},q')(\unb)\]
\end{remark}
\begin{remark}[Two-step extension of a \twoDash transducer]\label{twostep_extension}
	The \emph{Yoneda extension} of a functor \(F : \clX \to \clY\) with small domain and cocomplete codomain is the left Kan extension of \(F\) along the Yoneda embedding of \(\clX\); when \(\clX=\clA\times\opid\clQ\) and the codomain is of the form \(\clY=\freerig\clB\), one can perform a `two-step extension' of a transducer \(t : \clA\tdto\clB\), first extending \(t\) to a functor \(\clA^*\times\opid\clQ\to\freerig\clB\) monoidal in the \(\clA\) component, and then taking the Yoneda extension `in the variable \(\clA\)' (\ie leaving the \(\opid\clQ\) components alone), to the effect that each \twoDash transducer induces a functor
	\[\vxy{T : \freerig\clA\times\opid\clQ  \ar[r] &\freerig\clB}\]
	obtained as \(\Lan_{\yon_{\clA^*}} t\), which is monoidally cocontinuous in the \(\freerig\clA\) component. One could term this two-step operation sending \(t\) to \(T=\text{LY}(T)\) the \emph{Laurent-Yoneda extension} of \(t\).

	The Laurent-Yoneda extension \(t\mapsto \text{LT}(t)\) should be thought of as the Kleisli extension for a certain monad (each extension of \(t\) is the Kleisli extension with respect to a pseudomonad); perhaps this observation can be useful from a formal viewpoint, but we will not expand on it in the present work.
\end{remark}
\begin{defconstrunction}[The bicategory \(\TwoTDX\)]
	\label{definition_the_bicategory_bitwotran}
	Given all this, for fixed \(\clA,\clB\), there exists a category \(\TwoTDX(\clA,\clB)\) having
	\begin{enumtag}{t}
		\item\label{2t_1} objects the pairs \((\clQ,t)\) as in \eqref{transduc_eqn}; we term the category \(\clQ\) the `state category' of the \twoDash transducer;
		\item\label{2t_2} morphisms \((\clP,t) \to (\clQ,t')\) the `transducer transformations' \ie the pairs \((F,\theta)\) where \(F : \clP\to\clQ\) is a functor, and \(\theta : t\To t'(F,F)\), is a natural transformation, where \(t'(F,F)\) is defined as the composition
		\[\vxy[@C=1.5cm]{
			\clA\times\opid{\clP} \ar[r]^-{\clA\times \opid F}& \clA\times\opid{\clQ} \ar[r]^-{t'}& \freerig{\clB}.
			}\]
	\end{enumtag}
	Letting \(\clA,\clB\) vary, \(\TwoTDX(\firstblank,\firstblank)\) is a bifunctor \(\Cat^\op\times\Cat\to\Cat\);\footnote{Into large categories, but we will not pay much attention to this detail, and in general to size issues.} any two functors
	\[\vxy{\clC \ar[r]^-F & \clA & \clB \ar[r]^-G & \clD}\]
	will induce a functor \(\TwoTDX(F,G) : \TwoTDX(\clA,\clB) \to \TwoTDX(\clC,\clD)\) defined sending \((\clQ,t) : \clA\times\opid\clQ\to\freerig\clB\) to
	\[\vxy{
		\clC\times\opid\clQ \ar[r] & \clA\times\opid\clQ \ar[r]^-t & \freerig\clB \ar[r]^-{\Lan_{G^*}} &\freerig\clD
		}\]
	and where \(\Lan_{G^*}\) acts on a presheaf \(P : (\clB^*)^\op\to\Set\) sending it to the presheaf
	\[\und\mapsto \int^{\unb:\clB^*} P\unb\times\clD^*(\und,G^*\unb)\]
	where \(G^* : \clB^*\to\clD^*\) is the obvious extension of \(G\) to a functor between monoidal categories. All in all, for every \((\clQ,t) :  \TwoTDX(\clA,\clB)\), \(\TwoTDX(F,G)(\clQ,t)\) is defined as
	\[
		(\unc,q,q')\mapsto \int^{\unb:\clB^*} t(F\unc,q,q')(\unb)\times\clD^*(\und,G^*\unb).
	\]
	Moreover,
	\begin{itemize}
		\item given categories \(\clA,\clB,\clC\) we can define bifunctors
		      \[
			      \vxy{\firstblank\tranComp\firstblank : \TwoTDX(\clB,\clC) \times \TwoTDX(\clA,\clB) \ar[r] & \TwoTDX(\clA,\clC)}
		      \]
		      obeying the axioms of a (horizontal) composition law; this is defined, given \twoDash transducers \((s,\clQ) : \clA \tdto \clB\) and \((t,\clP) : \clB \tdto \clC\), as the \twoDash transducer \((\clP\times\clQ,T\circ (s\times\opid{\clP}))\) obtained from the composition
		      \[\label{compozia}\vxy[@C=1.5cm]{
			      \clA^* \times \opid{\clQ}\times\opid{\clP} \ar[r]^-{s\times\opid{\clP}} &
			      \freerig\clB \times \opid{\clP} \ar[r]^-T &
			      \freerig{\clC}\\
			      }\]
		      where \(T := \Yan[\clB] t\) is the two-step extension of \(t\) in the \(\clB\) component, as in \autoref{twostep_extension}. We can then write the composition \(t\circ s : \clA\tdto\clC\) in coend language as
		      \begin{align*}
			      \lambda\unc.\Big((t\tranComp s) (\una;(p,q),(p',q'))[\unc]\Big)
			       & =	\lambda\unc.\Big(T(s(\una,q,q')[\firstblank], p,p')[\unc]\Big) \\
			       & =	\lambda\unc.\Big(\tcomp sta[q][p][b][\clB][\unc]\Big).
		      \end{align*}
		\item In each \(\TwoTDX(\clA,\clA)\) there is a distinguished element, the \emph{identity} \twoDash transducer \(\iota : \clA\tdto\clA\) with state category \(\clI\) the terminal category (discrete on a singleton set), and structure constants \(\iota_a\) the representable functor at the singleton list \((a\kons \emptyList)\). More precisely,
		      \[\vxy{\iota : \clA^*\times\clI^\op\times\clI \ar@{=}[r] & \clA^* \ar[r] & \freerig{\clA}}\]
		      is the Yoneda embedding up to the isomorphism \(\clA^*\times\clI^\op\times\clI \cong \clA^*\).
	\end{itemize}
\end{defconstrunction}
This gives rise to a bicategory \(\TwoTDX\) of (small) categories, \twoDash transducers \(t : \clA\tdto\clB\) as 1-cells, and transducer transformations as above.

We sketch how to obtain the bicategory structural isomorphisms, using coend calculus.
The associator is built starting from composable 1-cells \(\clA\tdto[(\clP,w)]\clB\tdto[(\clQ,v)]\clC\tdto[(\clR,u)]\clD\), and obtained from the chain of isomorphisms
\begin{center}
	\bareadju{\begin{align*}
			((u\tranComp v)\tranComp w)(\una; (p,q,r), (p',q',r')) & \cong \tcomp{w}{(u\tranComp v)}{a}[p][(q,r)]                                                                                          \\
			                                                       & \cong\int^{\unb : \clB^*} w(\una, p,p')[\unb]\times \left(\tcomp vub[q][r][c][\clC]\right)                                            \\
			                                                       & \cong\int^{\unb : \clB^*}\kern-.75em\int^{\unc : \clC^*} w(\una, p,p')[\unb]\times \Big(v(\unb;q,q')[\unc] \times u(\unc, r,r')\Big)  \\
			(u\tranComp (v\tranComp w))(\una; (p,q,r), (p',q',r')) & \cong \tcomp{(v\tranComp w)}{u}{a}[(p,q)][r][c][\clC]                                                                                 \\
			                                                       & \cong\int^{\unc : \clC^*} \left(\tcomp wva[p][q]\right) \times u(\unc, r,r')                                                          \\
			                                                       & \cong\int^{\unc : \clC^*} \kern-.75em\int^{\unb : \clB^*} \Big(w(\una, p,p')[\unb]\times v(\unb;q,q')[\unc]\Big) \times u(\unc, r,r')
		\end{align*}}
\end{center}
The left and right unitors for a \twoDash transducer \(\clA\tdto[(\clQ,t)]\clB\) are induced by the isomorphisms
\begin{align*}
	(\iota_\clB\tranComp t)(\una;(*,q),(*,q)') & \cong \int^{\unb : \clB^*} t(\una;q,q')[\unb]\times \iota(\unb;*,*)      \\
	                                           & \cong t(\una;q,q').                                                      \\
	(t\tranComp \iota_\clA)(\una;(q,*),(q,*)') & \cong \int^{\una' : \clA^*} \iota_\clA(\una',*,*)[\una]\times t(a';q,q') \\
	                                           & \cong \int^{\una' : \clA^*} \clA^*(\una',\una)\times t(a';q,q')          \\
	                                           & \cong t(\una;q,q').
\end{align*}
Rote computation yields all the needed coherences (pentagon, unit axioms) in order to define the bicategory structure on \(\TwoTDX\).
\begin{remark}
	\label{2cell_theta_triangle}
	Clearly, a 2-cell \(\theta : (\clP,s)\To (\clQ,t)\) between \((\clP,s),(\clQ,t) : \clA\tdto\clB\) in \autoref{definition_the_bicategory_bitwotran}.\ref{2t_2} does nothing but fill the obvious triangle
	\[\vxy[@R=5mm]{
		\clA\times \opid{\clP} \drtwocell<\omit>{<3>\theta}
		\ar[dd]_-{\clA\times F^\op\times F}\ar[dr]^-t & \\
		& \freerig{\clB} \\
		\clA\times \opid{\clQ} \ar[ur]_-{t'}&
		}\]
	thus providing a family of cells, compatible with the category structure of \(\clA\), of shape
	\[\label{dbl_cell}\vxy{
		\clP\ar@{}[dr]|-{\theta(a,-)}\ar[r]|-@{|}^{t(a,-)}\ar[d]_-F & \clP\ar[d]^-F \\
		\clQ\ar[r]|-@{|}_{t'(a,-)} & \clQ
		}\]
	in the double category of categories (in the tight direction) and profunctors (in the loose direction). `Compatibility' means the obvious thing, in double categorical language: given a morphism \(u : a \to a'\) in \(\clA\), we have equalities of cells
	\[\vxy{
		\clP\ar@{=}[d]\ar[r]|-@{|}^{t_a}\ar@{}[dr]|-{t_u} & \clP\ar@{=}[d] & \clP\ar[d]_-F \ar[r]|-@{|}^{t_a} \ar@{}[dr]|-{\theta(a,-)} & \clP\ar[d]^-F \\
		\clP\ar[r]|-@{|}_{t_{a'}}\ar@{}[dr]|-{\theta(a',-)}\ar[d]_-F & \clP\ar[d]^-F\ar@{}[r]|= & \clQ\ar@{=}[d] \ar[r]|-@{|}_{t'_a} \ar@{}[dr]|-{t'_u}& \clQ\ar@{=}[d] \\
		\clQ \ar[r]|-@{|}_{t'_{a'}} & \clQ & \clQ \ar[r]|-@{|}_{t'_{a'}} & \clQ
		}\]
\end{remark}
\begin{remark}\label{whiskering_2cells}
	Note, in particular, that the whiskering operation is induced in the following ways (left and right relying on slightly different constructions), given a diagram of shape
	\[\vxy{
		\clX \ar[r]|@-{*}^-{(\clH,h)} & \clA \rrtwocell<\omit>{\kern1.5em(F,\theta)}\ar@/^1pc/[rr]|@-{*}^{(\clP,s)}\ar@/_1pc/[rr]|@-{*}_{(\clQ,t)} && \clB \ar[r]|@-{*}^-{(\clK,k)} & \clY.
		}\]
	\begin{itemize}
		\item Every 2-cell \((F,\theta) : (\clP,s)\To (\clQ,t)\) induces, by functoriality and by the universal properties of the relevant coends, a natural transformation
		      \[\vxy[@R=5mm]{
			      \freerig\clA\times \opid{\clP} \drtwocell<\omit>{<3>\Theta}
			      \ar[dd]_-{\freerig\clA\times F^\op\times F}\ar[dr]^-{\text{LY}(s)} & \\
			      & \freerig{\clB}  \\
			      \freerig\clA\times \opid{\clQ} \ar[ur]_-{\text{LY}(t)}&
			      }\]
		      with components induced by \(U(\una)\times\theta\) via
		      \begin{align*}
			      U(\una)\times s(\una,p,p')[\unb] & \to U(\una)\times t(\una,Fp,Fp')[\unb]                         \\
			                                       & \to \int^{\una' : \clA^*} U(\una')\times t(\una',Fp,Fp')[\unb] \\
			                                       & \cong \Yan t(U,Fp,Fp')[\unb],
		      \end{align*}
		      whence a unique way to define the whiskering \((F,\theta) * (\clX,h)\) as the pasting of 2-cells
		      \[\vxy[@R=5mm]{
			      \clX\times\opid{\clH} \times\opid{\clP} \ar[dd]_{\clX\times\opid{\clH}\times F^\op\times F}\ar[rr]^-{h\times \opid{\clP}}&&\freerig\clA\times \opid{\clP} \drtwocell<\omit>{<3>\Theta}
			      \ar[dd]_-{\freerig\clA\times F^\op\times F}\ar[dr]^-{\Yan s} & \\
			      &&& \freerig{\clB}  \\
			      \clX\times\opid{\clH} \times \opid{\clQ}\ar[rr]_-{h\times \opid{\clP}}&&\freerig\clA\times \opid{\clQ} \ar[ur]_-{\Yan t}&
			      }\]
		\item The whiskering \((\clK,k)*(F,\theta)\) is instead simply defined as the whiskering
		      \[\vxy[@R=5mm]{
			      \clA\times \opid{\clP} \drtwocell<\omit>{<3>\theta}
			      \ar[dd]_-{\clA\times F^\op\times F}\ar[dr]^-s & \\
			      & \freerig{\clB} \ar[r]^-K&\VCat\big(\opid{\clK},\freerig\clY\big) \\
			      \clA\times \opid{\clQ} \ar[ur]_-{t}&
			      }\]
		      where \(K\) is the transpose of \(\text{LY}(k) : \freerig\clB\times\opid{\clK} \to \freerig\clY\).
	\end{itemize}
	It is straightforward to verify that left and right whiskering are compatible, \ie
	\[(\clK,k)*\big((F,\theta) * (\clX,h)\big) = \big((\clK,k)*(F,\theta)\big) * (\clX,h).\]
\end{remark}
\begin{remark}\label{natural_embed}
	There is a natural way to regard a profunctor \(p : \clA \pto\clB\) as a 2-transducer \((\clI,\bar p) : \clA\tdto\clB\), where \(\clI\) is the terminal category; \(\bar p\) is the unique profunctor making the triangle
	\[\vxy[@R=5mm@C=5mm]{
			\clA\times\opid\clI\times(\clB^*)^\op \ar[dr]^{\bar p}& \\
			& \Set \\
			\clA\times\clB^\op\ar[ur]_p \ar[uu]& \\
		}\]
	commutative, and constant at the empty set \(\varnothing\) for every other summand \(\clB^n\) of \(\clB^*\), for \(n\ge 2\).
	This straightforward construction provides a natural embedding \(\Prof\subseteq\TwoTDX\) (or an embedding of double categories \(\DProf\subseteq\DTDX\)). It is immediate to check that the embedding is compatible with the `Laurent-Yoneda' extension of \autoref{twostep_extension}.
\end{remark}
\begin{remark}[Many versions of a category]\label{remark_many_versions_of_a_category}
	We can describe the category \(\TwoTDX(\clA,\clB)\) in many different ways; each of these will prove useful in different contexts, which we mention while we outline the equivalence. All the following categories are equivalent.
	\begin{enumtag}{tc}
		\item \label{tdx_char_1} The category of functors of type \(\clA \times \opid{\clQ}\to \freerig{\clB}\), and the category of functors \(\clA^*\times \opid{\clQ}\to \freerig{\clB}\), strong monoidal in the variable \(\clA\), given the universal property of \(\clA^*\) (by currying, \(\clA \times \opid{\clQ}\to \freerig{\clB}\) corresponds to \(\clA \to \freerig{\clB}\emdash\Prof(\clQ,\clQ)\), and the codomain of the latter is monoidal under composition).
		\item \label{tdx_char_2} The category of functors \(\freerig\clA \times \opid{\clQ}\to \freerig{\clB}\), taking the two-step extension in the \(\clA\) component, and (by currying) the category of functors \(t : \opid{\clQ}\to \Cat(\freerig\clA,\freerig\clB)\), so that each \(t_{qq'} :\freerig\clA\to\freerig\clB\) is a left adjoint, by a well-known property of profunctors;
		\item \label{tdx_char_3} currying in another direction, \(\TwoTDX(\clA,\clB)\) is the category of profunctors from \(\clA\) to \(\clB\), enriched over the (monoidal) category \(\Set^{\opid\clQ}\) of endoprofunctors on \(\clQ\).
		\item Currying in yet another direction, \(\TwoTDX(\clA,\clB)\) can be seen as the category \(\Cat(\opid{\clQ},\Set^{\clA^*\times (\clB^*)^\op})=\Set^{\clA^*\times (\clB^*)^\op}\emdash\Prof(\clQ,\clQ)\), where the codomain \(\Set^{\clA^*\times (\clB^*)^\op}\) is monoidal when equipped with Day convolution (\(\clA^*\times (\clB^*)^\op\) is monoidal in the `obvious' way, \((\una,\unb)\otimes (\una',\unb') := (\una\concat\una',\unb\mathbin{\concat^\op}\unb')\)).
		\item \label{tdx_char_33} \(\TwoTDX(\clA,\clB)\) is also equivalent to the category \(\Psd(\Sigma\clA^*,\freerig{\clB}\emdash\Prof)\) of pseudofunctors
		\[\vxy{\bst : \Sigma\clA^* \ar[r] & \freerig{\clB}\emdash\Prof}\]
		where \(\Sigma\clA^*\) is the one-object bicategory associated to \(\clA^*\); the unique object is mapped by \(\bst\) to \(\clQ\), and the strong monoidal functor \(t\) of \autoref{definition_2_transducer} corresponds, evidently, to the action of \(\bst\) on 1- and 2-cells.
		\item \label{tdx_char_4} \(\TwoTDX(\clA,\clB)\) is also equivalent to the category of profunctors \(\clA^*\times\clQ\pto\clQ\times\clB^*\); note that a profunctor \(\clA^*\times\clQ\pto\clQ\times\clB^*\) corresponds to a profunctor \(\clA^*\times\opid{\clQ}\pto\clB^*\); such a map looks like a coKleisli arrow \(K_\clQ(\clA^*)\pto\clB^*\), if one considers a putative parametric comonad
		\[\vxy[@R=0cm]{
				K : \Cat\times\Cat \ar[r] & \Cat\\
				(\clQ,\clA) \ar@{|->}[r] & \clA\times\opid{\clQ}.
			}\]
		Hence, the last equivalence is between \(\TwoTDX(\clA,\clB)\) as defined in any other of the previous ways, and
		\item \label{tdx_char_6} the `graded coKleisli category' in the sense of \cite{Gaboardi2021} of the graded comonad \(K : \clA\mapsto \clA\times\opid{\clQ}\), lifted from categories to profunctors.
	\end{enumtag}
	Some of these characterizations are evident: \ref{tdx_char_1} is the definition; \ref{tdx_char_2} is what allowed us to define the LY extension in \autoref{twostep_extension}. \ref{tdx_char_33} and others are expanded a bit in \autoref{tdx_properties_and_corollaries} below.

	The two latter characterizations require more explanation. We refer to \cite{Gaboardi2021} for additional details, but we will expand on the construction of the graded coKleisli category of a graded comonad in \autoref{locally_graded_coKleisli}, when we will have to characterize adjunctions in \(\TwoTDX\).

	First, there exists a graded (2-)comonad obtained as composite of the comonad \(\Delta : \clQ\mapsto\opid{\clQ}\) and the coreader comonad, in the sense that we have a composite oplax functor
	\[\vxy{
			K : \Sigma(\Cat,\times) \ar[r]^-\Delta &\Sigma(\Cat,\times) \ar[r]^-W & \Cat
		}\]
	defining a \((\Cat,\times)\)-graded comonad \(K_\clQ : \clC\mapsto\clC\times\opid{\clQ}\). (The functor \(\Delta\)	is the one that sends a category \(\clQ\) to the category \(\opid{\clQ}\); it is strong monoidal, in fact a comonad on \(\Cat\).)
	The functor \(K_\clQ\) now lifts to \(\Prof\)
	\[\vxy{
		\Prof \ar@{.>}[r]& \Prof\\
		\Cat \ar[u]\ar[r]_-{K_\clQ} & \Cat\ar[u]
		}\]
	via a distributive law with the presheaf construction,
	having components (denote \(\bsP\clC:=\psh\clC\) for brevity)
	\[\vxy{
		\bsP{\clC}\times\opid\clQ \ar[r]^-{\bsP\clC\times\yon}& \bsP{\clC}\times\bsP{(\opid\clQ)} \ar[r]^-{\tau_{\clC,\opid\clQ}}& \bsP(\clC\times\opid\clQ)
		}\]
	defined in terms of the Yoneda embedding and the tensorial strength of \(\bsP\).

	Clearly, a \twoDash transducer is a coKleisli map \(K_\clQ\clA\pto\clB\), but only if we interpret the coKleisli construction exactly in the sense of \cite{Gaboardi2021}.
\end{remark}
\begin{corollary}\label{tdx_properties_and_corollaries}
	Some of the equivalent characterizations above make it easier to establish certain properties of \(\TwoTDX(\clA,\clB)\), they have interesting corollaries or suggestive interpretations: we collect a few here and in \autoref{linear_proalgebra}, \autoref{nota_2_fun} below. Some easy consequences are that \(\TwoTDX(\clA,\clB)\) admits all small colimits, by virtue of \ref{tdx_char_3}. Furthermore, from \ref{tdx_char_33} it follows that, at least informally and over bases of enrichment where a Grothendieck construction for presheaves is available, \(\TwoTDX(\clA,\clB)\) can be thought of as the category of \emph{Conduché functors} over categories of the form \(\Sigma\clA^*\), where fibres are \(\freerig\clB\)-enriched.\footnote{Admittedly, this is a bit imprecise, but the appearance of a Conduché property doesn't seem random here. A connection between automata theory and the Conduché property has been observed, in passing and using a different terminology, in \cite[§2.3]{Mellis2025}, and before that it has been employed in \cite{Kasangian2010} to study concurrency and bisimulation, linking process semantics to functorial factorisation properties; the latter reference is in particular focused with studying a Conduché property for functors between categories enriched in a locally posetal 2-category obtained from a free monoid \(A^*\) with respect to the prefix order.}

	Our choice to focus on \(\Set\) as base of enrichment is motivated by the fact that here we acquire this additional point of view.
\end{corollary}
It has just been observed that each hom-category \(\TwoTDX(\clA,\clB)\) admits all small colimits; furthermore, both pre- and post-composition are easily seen to be colimit-preserving; one can then suspect that each functor \(\firstblank\circ s\) and \(t\circ\firstblank\) has a right adjoint (respectively, the right extension along \(s\), and the right lifting along \(t\)). We now construct right extensions and lifts explicitly, mimicking what happens in the bicategory of profunctors.
\begin{remark}[Construction of right extensions and lifts in \(\TwoTDX\)]
	Consider two composable \twoDash transducers \((s,\clQ) : \clA \tdto \clB\) and \((t,\clP) : \clB \tdto \clC\) as in \eqref{compozia} above; there is an evident bijective correspondence between 2-cells \((F,\theta)\) of type \(t\tranComp s\To r\) and 2-cells \((\hat F,\hat\theta)\) of type \(s\To \langle t/r\rangle\),
	\[\footnotesize\vxy[@R=5mm]{
		\clA^*\times\opid{(\clQ\times\clP)} \ar@{->}[rd]^{t\circ s}\drtwocell<\omit>{<3>\theta} \ar@{->}[dd]_{\clA^*\times\opid F} &  & \clA^*\times\opid\clQ \ar@{->}[rd]^s \ar@{->}[dd]_{\clA^*\times\opid{\hat F}}\drtwocell<\omit>{<3>\hat\theta} &  \\
		& \freerig\clC \ar@{}[r]|(.4)\cong& & \freerig\clB \\
		\clA^*\times\opid\clN \ar@{->}[ru]_r &  & \clA^*\times\opid{(\clN^\clP)} \ar@{->}[ru]_-{\langle t/r\rangle} &
		}\]
	constructed as follows. Given a 2-cell \((F,\theta)\) on the left, the functor \(\hat F\) is evidently the mate \(\clP\to\clN^\clQ\) of \(F\), and \(\hat\theta\) is determined under the chain of bijections

	\bareadju{
		\begin{align*}
			\freerig\clC((t\circ s)(\una,pqp'q'),r(\una,Fpq,Fp'q')) & \cong \int_{\unc:\clC^*} \Set\Big(\tcomp sta[q][p][b][\clB][\unc],r(\una,Fpq,Fp'q')[\unc]\Big)                                        \\
			                                                        & \cong \int_{\unc:\clC^*}\int_{\unb:\clB^*}	\Set\big(s(\una,q,q')[\unb]\times t(\unb;p,p')[\unc],r(\una,Fpq,Fp'q')[\unc]\big)          \\
			                                                        & \cong \int_{\unc:\clC^*}\int_{\unb:\clB^*}	\Set\big(s(\una,q,q')[\unb] , \big\{ t(\unb;p,p')[\unc],r(\una,Fpq,Fp'q')[\unc]\big\}\big) \\
			                                                        & \cong \int_{\unb:\clB^*}	\Set\Big(s(\una,q,q')[\unb] , \int_{\unc:\clC^*}\big\{ t(\unb;p,p')[\unc],r(\una,Fpq,Fp'q')[\unc]\big\}\Big)
		\end{align*}}
	which is exactly a natural transformation
	\[\vxy{s(\una,q,q')\ar[r] & \langle t/r\rangle(\una,\hat Fq,\hat Fq')}\]
	if \(\hat Fq=F(-,q)\) and we define \(\langle t/r\rangle : \clA^*\times\opid{(\clN^\clP)} \to \freerig{\clB}\) as sending \((\una,W,W') :  \clA^*\times\opid{(\clN^\clP)}\) to the presheaf on \(\clB^*\)	given by the end
	\[\unb\mapsto\int_{\unc:\clC^*}\big\{ t(\unb;p,p')[\unc],r(\una,Wp,Wp')[\unc]\big\}.\]
	Reasoning similarly, one determines a bijective correspondence between cells of the two types
	\[\vxy[@R=2mm]{
		\star \ar@{->}@/^.5pc/[rd]^{t\circ s}\drtwocell<\omit>{<2>\theta} \ar@{->}[dd] &  & \star \ar@{->}@/^.5pc/[rd]^t \ar@{->}[dd]\drtwocell<\omit>{<2>\tilde\theta} &  \\
		& \freerig\clC & &  \freerig\clC\\
		\star \ar@{->}@/_.5pc/[ru]_r &  & \star \ar@{->}@/_.5pc/[ru]_-{\langle r/s\rangle} &
		}\]
	where the right extension \(\langle r/s\rangle\) of \(r\) along \(s\) is defined as the end
	\[\int_{\unb:\clB^*}\big\{s(\una,q,q')[\unb],r(\una,Uq,Uq')[\unc]\big\}.\]
\end{remark}
We end this round of consequences of the many equivalent characterizations in \autoref{remark_many_versions_of_a_category} with a somewhat technical observation.
\begin{remark}\label{nota_2_fun}
	From \ref{tdx_char_2} above it follows that every hom-category of \(\TwoTDX\) is the total category of a fibration over \(\Cat\), because the assignment \(\bbP_{\clA,\clB}=\Prof_{\Set^{\clA^*\times (\clB^*)^\op}}\) sending \(\clQ\) to the category of endoprofunctors on \(\clQ\), enriched over \(\Set^{\clA^*\times (\clB^*)^\op}\) is functorial (and contravariant).

	So, we can think of \(\TwoTDX(\clA,\clB)\) as the Grothendieck-Bénabou construction of \(\bbP_{\clA,\clB}\).

	As a consequence, \(\TwoTDX(\clA,\clB)\) is small complete and cocomplete, and the assignment
	\[\vxy{(\clA,\clB) \ar@{|->}[r]& \int_\clQ \Big(\Prof_{\Set^{\clA^*\times (\clB^*)^\op}}\Big) \cong \TwoTDX(\clA,\clB)}\]
	\emph{qua} (pseudo)promonad, factors through the domain functor \(\text{dom}\) yielding the total category of a fibration, \ie
	\[\label{liftaggio_at_fibs}\vxy{
		&& \FIB/\Cat\ar[d]^{\text{dom}} \\
		\Cat^\op\times\Cat\ar[rr]_-{\TwoTDX(-,-)} \ar[urr]&& \CAT
		}\]
	This means that every hom-category \(\TwoTDX(\clA,\clB)\) is a pseudofunctor \(\Cat^\op\to\CAT\), \ie that \(\TwoTDX\) is a \emph{\(\Psd(\Cat^\op,\CAT)\)-enriched 2-category}, or (for lack of a better name) a `graded 2-category'. Treating it as a 3-category --after all, \(\FIB/\Cat\) has 2-cells!-- is however not completely straightforward: \(\TwoTDX(\clA,\clB)\) is not a 2-category (and \(\TwoTDX\) not a 3-category), because the functor \(\clQ\mapsto\opid\clQ\) is not a 2-functor (what would be its covariance type on 2-cells?).
\end{remark}
Instead, let's focus on a series of parallels with linear algebra, inspired by the discussion so far.
\begin{example}[`Linear algebra' in \(\TwoTDX\)]\label{linear_proalgebra}
	Let \(k\) be a field. A \(k\)-vector space \(V\) equipped with an endomorphism \(f : V\to V\) corresponds to a \(k[X]\)-module structure on \(V\) (by the universal property of \(k[X]\) as the free \(k\)-algebra on one generator); here we would like to argue that such a notion is categorified by \twoDash transducers \(\clI\tdto\varnothing\) so that, if by `linear map between vector spaces' we choose to mean `colimit-preserving functor between cocomplete categories', \(\clL=\TwoTDX(\clI,\varnothing)\) can be thought of as a category of categories equipped with an endoprofunctor.

	Indeed \autoref{definition_2_transducer} yields that a \twoDash transducer \(\clI\tdto\varnothing\) consists of a pair \((\clQ,t)\) where \(\clQ\) is a category and \(t\) a functor of type
	\[\vxy{\clI \times \opid\clQ\ar[r] & \freerig\varnothing}\]
	which (using the fact that \(\clI^*=\bbN,\freerig\varnothing=\Set\)) corresponds to a single profunctor \(t=t_* : \opid\clQ\to\Set\), together with all its iterates \(t\circ\dots\circ t\). Now \ref{tdx_char_2} yields that a \twoDash transducer \(\clI\tdto\varnothing\) also extends uniquely to a unique functor
	\[\vxy{
			\Set/\bbN\times\opid\clQ\ar[r] & \Set
		}\]
	which is a \((q,q')\)-wise a left adjoint (the presheaf category of \(\bbN\) regarded as discrete category is the slice \(\Set/\bbN\)): so \(T\) is in this case just given by the `polynomial of endofunctors' sending \(((S_n\mid n : \bbN),q,q') : \Set/\bbN\times\opid\clQ\) to
	\[T((S_n),q,q') = \sum_{n:\bbN}S_n\times t^n(q,q').\]
	Compare this expression with the representation of how a linear endomorphism \(T : V\to V\) of a \(k\)-vector space \(V\) yields a \(k[X]\) representation acting with a polynomial \(g(X) = \sum_i \lambda_i X^i\) on a vector \(v :  V\) as \(\sum_i \lambda_i T^i(v)\), \cf \cite[VIII.7.1]{grillet}.
\end{example}
By a similar token, more generally, the category \(\TwoTDX(\clA,\varnothing)\) can be thought of as a category of pairs \((\clQ,\{t_a\})\) where \(t : \clA\times\opid\clQ\to\Set\) is an \(\clA\)-indexed profunctor. This somehow categorifies the well-known notion of a (\(A\)-labeled, nondeterministic) transition system, usually presented as a function \(A\times Q\to \Bool^Q\) (or more conveniently, as a coalgebra for the functor \(\Bool^{\blank\times A}\)). The previous discussion substantiates the idea that hom-categories \(\TwoTDX(\clA,\clB)\) have a rather direct combinatorial interpretation when \(\clA,\clB\) are `very small'.
\begin{remark}[\twoDash transducers between categories of very low cardinality]
	\label{transducers_low_cardinality}
	Let \(\clI\) be the terminal category, and \(\clM=\TwoTDX(\clI,\clI)\) the hom-category of transducers \(\clI\tdto \clI\); clearly, \(\clM\) is monoidal with respect to composition, and \(\clL\) above is a \(\clM\)-bimodule. Equally clearly, it is also an \(\clN=\TwoTDX(\varnothing,\varnothing)\)-left module, but \(\clN\) has way less structure and its action is quite more trivial.\footnote{In more detail: \(\clN\) has objects the pairs \((\clQ,\hom_\clQ)\) of a category plus its identity profunctor, and morphisms all the functors \(F : \clP\to\clQ\); as such, it is monoidally equivalent to \(\VCat\). The composition as defined in \autoref{compozia} is easily seen to be isomorphic to the projection \(\clN\times\clL\to\clL\).}

	The category \(\TwoTDX(\clI,\clI)\) is easily seen to consist of pairs \((\clQ,s : \opid\clQ\to \Set/\bbN)\), and thus the composition \(\clI \tdto[(\clQ,s)]\clI \tdto[(\clP,t)]\varnothing\) as outlined in \autoref{compozia} boils down to the functor
	\[\label{compo_rule_when_domain_small}\vxy[@R=0cm]{
		\opid\clQ\times\opid\clP \ar[rr]^-{s\times\opid\clP} && \Set/\bbN\times\opid\clP \ar[r]^-T & \Set
		}\]
	sending \((q,q',p,p')\) to \(\sum_{n:\bbN} s(q,q')_n\times t^n(p,p')\).
	It is also interesting to work out what the composition map
	\[\vxy{
			\TwoTDX(\clI,\varnothing)\times\TwoTDX(\varnothing,\clI) \ar[r] & \TwoTDX(\varnothing,\varnothing)
		}\]
	boils down to given these characterizations; the result must reduce to the hom-functor of a category, but which one? Consider a diagram of shape
	\[\label{nbdy_expect_free_promonad}\vxy{
		\varnothing \ar[r]|@-{*}^{(\clQ,s)}& \clI \ar[r]|@-{*}^{(\clP,t)}& \varnothing
		}\]
	the \twoDash transducer \((\clQ,s)\) is just the hom-functor of a category \(\clQ\), enriched over \(\Set/\bbN\) in the trivial way, \ie describing the \(\bbN\)-graded set constant at \(\clQ(q,q')\). So, the composition in \eqref{compo_rule_when_domain_small} reduces to
	\begin{align*}
		T((S_n),q,q') & \cong \sum_{n:\bbN}s(q,q')_n\times t^n(p,p')                                            \\
		              & \cong \clQ(q,q')\times\sum_{n:\bbN} t^n(p,p')                                           \\
		              & \cong \clQ(q,q')\times \clP_{t^*}(p,p') = (\clQ\times\clP_{t^*})\big((q,p),(q',p')\big)
	\end{align*}
	where \(\clP_{t^*}\) is the category obtained from the free promonad \(\sum_n t^n\) on \(\clP\), and \(\clQ\times\clP_{t^*}\) the product of categories.\footnote{%
	Last, one can consider a composition of cells of the same type, but in the opposite direction, \(\vxy{\clI \ar[r]|@-{*}^{(\clP,t)}& \varnothing \ar[r]|@-{*}^{(\clQ,s)}& \clI}\),	yielding a functor of type \(\opid\clQ\to\Set/\bbN\) defined as the trivial enrichment of \(t(p,p')\) over \(\Set/\bbN\).%
	}
\end{remark}
\begin{remark}
	What interpretation does the above construction have? Among different analogies one can find for \eqref{compo_rule_when_domain_small} in the realm of linear algebra, the following seems the most streamlined: let \(k\) be a field; let \(\bbN\pitchfork k\cong \prod_{n : \bbN}k\) be the \emph{power} of \(k\) by \(\bbN\), regarded as a \(k\)-algebra.

	Every given linear operator \(T : V\to V\) of an \(n\)-dimensional \(k\)-module \(V\) is in an evident sense a matrix \([n]\times[n] \to k\), and a sequence of endomorphisms \(A_n : W\to W\) of another finite-dimensional \(k\)-module \(W\) can be regarded as a single matrix \([m]\times [m]\to\bbN\pitchfork k\); then, one can consider the linear operator \(W\otimes V\), defined as
	\[\sum_{n\ge 0} A_n\otimes T^n : w\otimes v\mapsto \sum_{n\ge 0} A_nw\otimes T^nv\label{serione}\]
	provided the sum makes sense (either because \(A_n\) is almost-everywhere zero sequence, or because \(T\) is nilpotent, or because one can equip \(V,W\) with topologies in which \eqref{serione} converges). This can be seen as an element of \(\text{End}_k(W)\llbracket T\rrbracket\) in a suitable sense. If \(1\le q,q'\le m\), \(1\le p,p'\le n\) the matrix element of \(\sum_{n\ge 0} A_n\otimes T^n\) at the entry \(((p,q),(p',q'))\) looks precisely \(\sum_{n\ge 0} (A_n)_{qq'}\otimes (T^n)_{pp'}\).
\end{remark}
Let's now turn to the definition of the (posetal) bicategory of 1-transducers, and conclude the section, before turning to the study of categorical properties of \(\TwoTDX\).
\begin{remark}\label{bic_of_1_tnd}
	The bicategory \(\TDX\) of 1-transducers is defined specializing the above definition to the case where the base of enrichment is the Boolean algebra \(\Bool=\{0,1\}\) regarded as a Cartesian closed category. \(\TDX\) has
	\begin{itemize}
		\item objects the sets \(A,B,C,D,\dots\);
		\item 1-cells \(A\tdto B\)	the 1-transducers \((Q, t : A^* \to M(Q,\malg[\Bool]{B}))\), \ie the functions of type
		      \[\vxy{
				      t : A^*	\times Q\times Q \ar[r] & \Bool^{B^*}=\malg[\Bool]{B}
			      }\]
		\item 2-cells \(f : (P,s)\To (Q,t)\) the functions \(f : P\to Q\) between carriers such that
		      \[\forall(\una,p,p') : s(\una,p,p')\le t(\una,fp,fp')\]
		      (the inequality has to be interpreted pointwise, or rather, as the set-theoretic inclusion of the subset of which \(s(\una,p,p')\) is the characteristic function into the subset of which \(t(\una,fp,fp')\) is the characteristic function).
	\end{itemize}
	Composition of 1-cells \(A \tdto[(P,s)] B\tdto[(Q,t)] C\) is defined through the composition of relations as
	\[\vxy[@C=5mm]{
			A^*\times (P\times Q)\times (P\times Q)\ar@{=}[r] & (A^*\times P\times P)\times Q\times Q \ar[r]^-s & \Bool^{B^*} \times Q\times Q \ar[r]^-T & \Bool^{C^*}\\
		}\]
	where \(T\) is the relation induced by the universal property of \(\Bool^{B^*}\) as free quantale over \(B\), given that \(M(Q, \Bool^{C^*})\) is a quantale.

	Composition of 2-cells \(P \xto f Q\xto g R\) ends up being simply witnessed by the inclusion
	\[\forall(\una,p,p') : r(\una,p,p')\le s(\una,fp,fp')\le t(\una,gfp,gfp').\]
\end{remark}
\section{Properties of \(\TwoTDX\)}\label{properties}
The present section is the core of our work, focusing on the properties of 2-transducers as defined in \autoref{definition_the_bicategory_bitwotran}: we start by studying completeness and cocompleteness	properties of 2-transducers (in \autoref{sec_completeness}); we continue assessing the existence of monoidal structures (in \autoref{sec_monoidality}), and then we classify adjunctions (there are very few; \cf \autoref{very_few}), while describing monads inside our bicategory (in \autoref{whatsa_monad}); lastly, we focus on the relations with other bicategories of automata (in \autoref{sec_relations-automata}).
\subsection{Completeness}\label{sec_completeness}
The existence of limits and colimits in \(\TwoTDX\) could be inferred from some of the equivalent descriptions for \(\TwoTDX(\clA,\clB)\) given in \autoref{remark_many_versions_of_a_category}.

However, it is more natural to study a \emph{double} category \(\DTDX\) in the way hinted at in \autoref{2cell_theta_triangle}, as this latter presentation makes it easier to assess its cocompleteness and often (when there are tabulators in the double category of \(\clV\)-profunctors) also its completeness; on the contrary, the bicategory spanned by globular cells inherits at best only lax limits and colimits of lax functors (\cf \cite[15.18]{Garner2016}).

Let's make precise the definition of such a double category, and subsequently we will proceed to construct limits and colimits in \(\DTDX\), proving they all exist (more precisely, we construct products in \autoref{products_in_TwoTDX}, equalizers in \autoref{equalizers_in_TwoTDX}; the construction of colimits suffers the complexity of colimits in \(\Cat\): it is easy to construct coproducts in \(\DTDX\), we do so in \autoref{coproducts_in_TwoTDX}; it is more difficult to exhibit a direct construction of coequalizers, but in \autoref{coequalizers_in_TwoTDX} we prove that one can recycle the argument usually given to build colimits of profunctors, \cf \cite[§6.3]{grandispare1999limits}). Cotabulators exist, as we prove in \autoref{cotabs}, at least for bases of enrichment where profunctors have them; tabulators do not, and the reason is more intrinsic; in \autoref{not_all_tabs} we outline it.
\begin{definition}[The double category of transducers]\label{dbcat_of_tdx}
	The double category \(\DTDX\) of \twoDash transducers has
	\begin{itemize}
		\item objects are small \(\clV\)-categories \(\clA,\clB\), etc.;
		\item a tight cell \(F : \clA\to\clB\) is a \(\clV\)-functor;
		\item a loose cell \((\clQ,t) : \clA \tdto \clB\) is a profunctor as in \ref{tdx_char_3}, \ie of type
		      \[\vxy{
				      t : \clA \times \opid{\clQ} \times (\clB^*)^\op \ar[r] & \clV
			      }\]
		\item a cell  \((U,\alpha)\) with frame
		      \[\vxy{
			      \clA \ar[d]_F \ar[r]|-@{*}^-{(\clQ,s)}& \clB\ar[d]^G \\
			      \clA' \ar[r]|-@{*}_-{(\clP,t)}& \clB'
			      }\]
		      consists of a pair where \(U : \clQ\to\clP\) is a functor and \(\alpha\) is a natural transformation with components
		      \[\vxy{
				      \alpha : s(a,q,q')(b) \ar[r] & t(Fa,Uq,Uq')(Gb),
			      }\]
		      \ie a natural transformation filling the diagram
		      \[\vxy{
				      \clA \times \opid{\clQ} \times \clB^\op \ar[dr]^s\ar[dd]_{F\times\opid U\times G}\drtwocell<\omit>{<4>\alpha}&\\
				      & \clV\\
				      \clA' \times \opid{\clP} \times (\clB')^\op \ar[ur]_t
			      }\]
	\end{itemize}
\end{definition}
The construction of some tight limits and colimits mirrors the way in which they are already constructed in the double category \(\DProf\) of profunctors, \cf \cite[6.3]{grandispare1999limits}. The only subtlety arises in taking care of the grading components \(\clQ\) in a family of 1-cells forming a co/cone.

In the following we sketch the construction of limits and colimits that are building blocks for all (tight, and double) co/limits, or we provide counterexamples to their existence. Moreover, from now on, we focus solely on the case where \(\clV=\Set\).
\begin{construction}[Coproducts in \(\DTDX\)]\label{coproducts_in_TwoTDX}
	Coproducts in \(\DTDX\) are a bit more involved to define, as the usual recipe for building them in the double category of profunctors does not work out of the box. However, it's easy to adapt the same argument.

	Recall, for convenience, the universal property of (binary) double coproducts: given \(t : \clA\tdto\clB\) and \(s : \clC \tdto\clD\) we have to find cells
	\[\label{coprod_injezioni}\vxy{
		\clA \ar@{}[dr]|-{\sm[r]{I_s & \iota_s}}\ar[d]\ar[r]|-@{*}^-s & \clB \ar[d]& \clC\ar@{}[dr]|-{\sm[r]{I_t & \iota_t}} \ar[d]\ar[r]|-@{*}^-t & \clD\ar[d] \\
		\clA+\clC \ar[r]|-@{*}_-{s\oplus t} & \clB+\clD & \clA+\clC \ar[r]|-@{*}_-{s\oplus t} & \clB+\clD
		}\]
	with loose codomain \(s\oplus t : \clA+\clC \tdto \clB + \clD\) a \twoDash transducer between the coproducts in \(\Cat\), so that for any other pair of cells
	\[\vxy{
		\clA \ar@{}[dr]|-{\sm[r]{H_s & \theta_s}}\ar[d]\ar[r]|-@{*}^-s & \clB \ar[d]& \clC \ar@{}[dr]|-{\sm[r]{H_t & \theta_t}}\ar[d]\ar[r]|-@{*}^-t & \clD\ar[d] \\
		\clX \ar[r]|-@{*}_-r & \clY & \clX \ar[r]|-@{*}_-r & \clY
		}\]
	factors as
	\[\vxy{
		\clA \ar@{}[dr]|-{\sm[r]{I_s & \iota_s}}\ar[d]\ar[r]|-@{*}^-s & \clB \ar[d]& \clC \ar@{}[dr]|-{\sm[r]{I_t & \iota_t}}\ar[d]\ar[r]|-@{*}^-t & \clD\ar[d] \\
		\clA+\clC \ar[d]\ar@{}[dr]|-{\sm[r]{\bar H & \bar\theta}}\ar[r]|-@{*}_-{s\oplus t} & \clB+\clD \ar[d]& \clA+\clC \ar@{}[dr]|-{\sm[r]{\bar H & \bar\theta}}\ar[d]\ar[r]|-@{*}_-{s\oplus t} & \clB+\clD\ar[d]\\
		\clX \ar[r]|-@{*}_-r & \clY & \clX \ar[r]|-@{*}_-r & \clY
		}\]
	for a unique cell \((\bar H,\bar\theta)\). In order to find \(s\oplus t\) as above, we leverage on the existence of a canonical map
	\[\vxy{\gamma : \freerig\clB + \freerig\clD \ar[r] & \freerig{(\clB+\clD)}}\]
	obtained	applying the functor \(\freerig{\blank}\) to the coproduct cospan \(\clB \xto{i_\clB}\clB+\clD\xot{i_\clD}\clD\). More explicitly, the left arrow is obtained as composition
	\[\vxy{
		\Set^{(\clB^*)^\op} \ar[r]^-{\Lan_{i_\clB^*}}& \Set^{(\clB^*+\clD^*)^\op} \ar[r] &
		\Set^{(\clB^*\oast\clD^*)^\op}\ar@{=}[r]^\sim &
		\Set^{((\clB+\clD)^*)^\op}
		}\]
	and similarly we do for the right arrow; note that by \(\oast\) we mean the coproduct (`free product') of monoidal categories, which has a canonical map from the coproduct of \(\clB^*,\clD^*\) \emph{in \(\Cat\)}. This defines a \twoDash transducer
	\[\vxy{
		(\clA+\clC)\times \opid{(\clP+\clQ)} \ar[r]^-{s\oplus t} & \freerig\clB + \freerig\clD \ar[r]^-\gamma & \freerig{(\clB+\clD)}
		}\]
	whose domain is of the form \(\clA\times\opid\clP + \clC\times\opid\clQ + (\text{higher terms})\); on the first two summands, we use the given transducer \(s\), and then embed in the coproduct \(\freerig{\clB} + \freerig\clD\), and similarly we do on the second summand. On every other summand, we pick the initial object of \(\freerig{(\clB+\clD)}\).

	This gives double cells as in \eqref{coprod_injezioni} which witness the double coproduct.
\end{construction}
\begin{construction}[Products in \(\DTDX\)]\label{products_in_TwoTDX}
	Upon reordering the factors, from a family of \twoDash transducers \((\clQ_i,t_i) : \clA_i\tdto\clB_i\) we obtain a \twoDash transducer between the products of all \(\clA_i\)'s and \(\clB_i\)'s,
	\[\vxy{
			\prod_{i :  I}\clA_i\times\opid{\big(\prod_{i :  I} \clQ_i\big)}\times\big(\prod_{i :  I}\clB_i^*\big)^\op\ar[r] & \prod_{i :  I}\Set
		}\]
	now, post-composing with the \(I\)-fold Cartesian product functor \(\prod_I : \Set^I \to\Set : (A_i)\mapsto \prod_{i :  I}A_i\), and pre-composing this with the canonical map \(\big(\prod_{i :  I}\clB_i\big)^*\to\prod_{i :  I}\clB_i^*\), we obtain a \twoDash transducer \(\prod_I t_i : \prod_{i :  I} \clA_i\tdto\prod_{i :  I} \clB_i\). Dually to the previous construction, the universal	property of double products is verified, at the level of cells, if every family of cells as on the left factors as on the right:
	\[\vxy[@R=7mm]{
		\clX \ar@{}[ddr]|{\sm[r]{H_i &\theta_i}}\ar[r]|-@{*}^{(\clP,s)}\ar[dd]& \clY \ar[dd] & \clX \ar[d]\ar[r]|-@{*}^{(\clP,s)}& \clY \ar[d]\\
		& & \prod_i \clA_i \ar@{}[dr]|{\sm[r]{\hat H & \hat\theta}}\ar[r]|-@{*}\ar[d]& \prod_i \clB_i\ar[d]\\
		\clA_i \ar[r]|-@{*}_-{(\clQ_i,t_i)}& \clB_i & \clA_i \ar[r]|-@{*}_-{(\clQ_i,t_i)}& \clB_i
		}\label{theta_i_for_prod}\]
	for a unique \(\sm[r]{\hat H & \hat\theta}\); unwinding the type of \(\hat\theta\), this is true (when taking \(\hat F,\hat G,\hat H\) to be the functors induced from \(\prod_i \clA_i,\prod_i \clB_i,\prod_i \clQ_i\) respectively) if and only if a unique cell
	\[\vxy{
			\hat\theta : s(x,p,p')(\uny) \ar[r] & \prod_{i :  I} t_i(F_i x,H_ip,H_ip')(G_i^*\uny)
		}\]
	is induced; this is the case, using the components of \(\theta_i\) as given in \eqref{theta_i_for_prod}.
\end{construction}
To conclude the discussion about colimits, we observe that one can construct reflexive coequalizers (and together with coproducts, constructed above, prove that \(\DTDX\) admits all tight colimits). A similar argument could be adapted to the existence of \emph{general} coequalizers, at the price of handling a more complicated shape of cocone at a certain step of the proof. In short, it is simpler to construct \emph{reflexive} coequalizers, because at some point the siftedness of the category \(\{\xymatrix{0 \arar{}{} & 1\ar[l]}\}\), although not strictly necessary, will turn out to shorten and simplify our task. This suffices to prove cocompleteness, as reflexive coequalizers and coproducts build all colimits.

\medskip
To construct reflexive coequalizers in \(\DTDX\), it will be crucial to employ the two following facts:
\begin{itemize}
	\item the explicit construction of coequalizers in the double category of profunctors, obtained as follows (\cf \cite[§6.3.b]{grandispare1999limits}): start with two cells
	      \[\label{qui}\vxy{
		      \clA_0\ar[r]|-@{*}^{s_0}\drtwocell<\omit>{\alpha_0}\ar[d]_f & \clB_0 \ar[d]^{f'}& \clA_0 \ar[r]|-@{*}^{s_0}\drtwocell<\omit>{\alpha_1}\ar[d]_g& \clB_0 \ar[d]^{g'}\\
		      \clA_1\ar[r]|-@{*}_{s_1} & \clB_1 & \clA_1\ar[r]|-@{*}_{s_1} & \clB_1
		      }\]
	      and use the tuples \((f,f',\alpha_0),(g,g',\alpha_1)\) to induce a parallel pair of functors
	      \[\vxy{
			      \collage{\clA_0}{s_0}{\clB_0}\ar[dr] \arar[rr] FG && \collage{\clA_1}{s_1}{\clB_1}\ar[dl]\\
			      &\{0\to 1\} &
		      }\]
	      Coequalize the pair \((F,G)\), which is `barreled' over \(\{0\to 1\}\), meaning that \(F,G\) make the triangle of maps into the interval commute;\footnote{André Joyal calls a \emph{barrel} an object of the category \(\Cat/\{0\to 1\}\); profunctors and barrels are equivalent through the collage construction sending a profunctor \(s : \clX\pto\clY\) to the map \(\collage\clX s\clY \to \{0\to 1\}\) having fibers \(\clX\) over \(0\) and \(\clY\) over \(1\).} clearly, the category \(\clZ\) so obtained also admits a functor over \(\{0\to 1\}\), hence it defines a profunctor. A little diagram chase shows that \(\clZ\) is of the form \(\collage\clA{p_\clZ}\clB\) if \(\clA,\clB\) are respectively the coequalizers of the pairs \((f,g)\) and \((f',g')\), for a certain profunctor \(p_\clZ\) between them.
	\item the fact (\cf \autoref{tdx_char_4}) that a transducer \((\clQ,t) : \clA\tdto\clB\) is just a profunctor of type
	      \[\vxy{
		      \clA\times\clQ \ar[r]|-@{|} &\clQ\times\clB^*.
		      }\]
\end{itemize}
We are now going to adapt, and repeat, the above argument.
\begin{construction}[Coequalizers in \(\DTDX\)]\label{coequalizers_in_TwoTDX}
	Start with two cells similar to the ones in \eqref{qui},
	\[\vxy{
		\clA_0\ar[r]|-@{*}^{(\clQ_0,s_0)}\ar@{}[dr]|{\sm[r]{H&\alpha_0}}\ar[d]_f & \clB_0 \ar[d]^{f'}& \clA_0 \ar[r]|-@{*}^{(\clQ_0,s_0)}\ar@{}[dr]|{\sm[r]{K&\alpha_1}}\ar[d]_g& \clB_0 \ar[d]^{g'}\\
		\clA_1\ar[r]|-@{*}_{(\clQ_1,s_1)} & \clB_1 & \clA_1\ar[r]|-@{*}_{(\clQ_1,s_1)} & \clB_1
		}\]
	where \(f,g : \clA_0 \rightrightarrows \clA_1\), \(H,K : \clQ_0 \rightrightarrows \clQ_1\) and \(f',g' : \clA_0 \rightrightarrows \clB_1\) all have a common section \(u,U\) and \(u'\) respectively.	Let's denote the coequalizers of said pairs in \(\Cat\) as the following diagrams
	\[\label{baubau}\vxy{
			\clA_0 \arar{f}{g}& \clA_1\ar[r] & \bar\clA & \clB_0 \arar{f'}{g'}& \clB_1\ar[r] & \bar\clB & \clQ_0 \arar HK & \clQ_1\ar[r] & \bar\clQ.
		}\]
	Repeating the argument valid in \(\DProf\), recalled above, one obtains a diagram of type
	\[\vxy{
			\collage {(\clA_0\times\clQ_0)}{s_0}{(\clQ_0\times\clB_0^*)}\arar FG & \collage {(\clA_1\times\clQ_1)}{s_1}{(\clQ_1\times\clB_1^*)}
		}\]
	where \(F\) is induced using \(f,f'\) and \(H\), and \(G\) is induced using \(g,g',K\) in the following way: in the notation of \autoref{cotabs}, denoting with \(j_\clX,j_\clY\) the collage inclusions of \(\clX,\clY\) into \(\collage\clX p\clY\),
	\begin{gather*}
		\begin{cases}
			F(j_{\clA_0\times\clQ_0}(a,q)) = j_{\clA_1\times\clQ_1}(fa,Hq) \\
			F(j_{\clQ_0\times\clB_0^*}(q,\unb)) = j_{\clQ_1\times\clB_0^*}(Hq,(f')^*\unb)
		\end{cases}\\
		\begin{cases}
			G(j_{\clA_0\times\clQ_0}(a,q)) = j_{\clA_1\times\clQ_1}(ga,Kq) \\
			G(j_{\clQ_0\times\clB_0^*}(q,\unb)) = j_{\clQ_1\times\clB_1^*}(Kq,(g')^*\unb)
		\end{cases}
	\end{gather*}
	furthermore, \(F,G\) have a common section, induced by the sections \(u,u',U\) above. From this definition, it follows at once that the projection from \(\collage {(\clA_1\times\clQ_1)}{s_1}{(\clQ_1\times\clB_1^*)}\) coequalizes \((F,G)\), whence a diagram
	\[\vxy{
			\collage {(\clA_1\times\clQ_1)}{s_1}{(\clQ_1\times\clB_1^*)}\ar[r]^-D  & \clC \ar[r]^-P & \{0\to 1\}
		}\]
	turning the coequalizer \(\clC\) of \((F,G)\) into a barrel, hence a profunctor. Now, the claim is that the fibers of \(P\) over \(0\) and \(1\) are respectively isomorphic to \(\bar\clA\times\bar\clQ\) and \(\bar\clQ\times\bar\clB^*\), notation as in \eqref{baubau}; there are unique maps \(D_0,D_1\) filling the diagrams
	\[\vxy{
			\clA_1\times\clQ_1\ar[d]_{D_0}\ar[r] & \collage {(\clA_1\times\clQ_1)}{s_1}{(\clQ_1\times\clB_1^*)} \ar[d]^D\\
			\clC_0 \ar[r]_{j_0} & \clC
		}\qquad
		\vxy{
			(\clQ_1\times\clB_1^*)\ar[d]_{D_1}\ar[r] & \collage {(\clA_1\times\clQ_1)}{s_1}{(\clQ_1\times\clB_1^*)} \ar[d]^D\\
			\clC_1 \ar[r]_{j_1} & \clC
		}\]
	it is easily seen that \(D_0,D_1\) coequalize the respective pairs in the diagrams
	\[\vxy{
		\clA_0\times\clQ_0 \arar{f\times H}{g\times K} & \clA_1\times\clQ_1 \ar[r]^-{D_0} & \clC_0 &
		\clQ_0\times\clB^*_0 \arar{(f')^*\times H}{(g')^*\times K} & \clQ_1\times\clB^*_1 \ar[r]^-{D_1} & \clC_1.
		}\]
	From here, one applies a completely routine argument to check that these are coequalizer diagrams; this is sufficient to prove the claim, since the siftedness of the indexing category for reflexive coequalizers entail that
	\[\bar\clA\times\bar\clQ \cong \text{coeq}\big(\xymatrix{
			\clA_0\times\clQ_0 \arar xy & \clA_1\times\clQ_1
		}\big)
		\qquad \bar\clQ\times\bar\clB^* \cong\text{coeq}\big(\xymatrix{
			\clQ_0\times\clB_0^* \arar xy & \clQ_1\times\clB_1^*
		}\big);\]
	more explicitly, the siftedness of \(\clJ=\{\xymatrix{0 \arar{}{} & 1\ar[l]}\}\) was sufficient to prove the claim exhibiting a cocone just for the diagonal arrows in
	\[\xymatrix{
			\clA_0\times\clB_0 \arar[d]{}{} \arar[r]{}{} & \clA_0\times\clB_1 \ar@{->}[r] \arar[d]{}{} & \clA_0\times\bar\clB \arar[d]{}{} \\
			\clA_1\times\clB_0 \ar@{->}[d] \arar[r]{}{} & \clA_1\times\clB_1 \ar@{->}[d] \ar@{->}[r] & \clA_1\times\bar\clB \ar[d]\\
			\bar\clA\times\clB_0 \arar[r]{}{} & \bar\clA\times\clB_1 \ar[r]& \bar\clA\times\bar\clB.
		}\]
	Instead of being derived from the existence of coproducts + reflexive coequalizers, coequalizers can be built directly trying to construct a cocone for the whole diagram above.
\end{construction}
\begin{construction}[Equalizers in \(\DTDX\)]\label{equalizers_in_TwoTDX}
	Start with two cells
	\[\label{eq_cells}\vxy{
		\clA_0\ar[r]|-@{*}^{(\clQ_0,s_0)}\ar@{}[dr]|{\sm[r]{H&\alpha_0}}\ar[d]_f & \clB_0 \ar[d]^{f'}& \clA_0 \ar[r]|-@{*}^{(\clQ_0,s_0)}\ar@{}[dr]|{\sm[r]{K&\alpha_1}}\ar[d]_g& \clB_0 \ar[d]^{g'}\\
		\clA_1\ar[r]|-@{*}_{(\clQ_1,s_1)} & \clB_1 & \clA_1\ar[r]|-@{*}_{(\clQ_1,s_1)} & \clB_1
		}\]
	and consider the equalizer diagrams
	\[\vxy{
			\clA \ar[r] & \clA_0 \arar{f}{g}& \clA_1 & \clB \ar[r]& \clB_0 \arar{f'}{g'}& \clB_1 & \clQ\ar[r] & \clQ_0 \arar HK & \clQ_1.
		}\]
	Observe that applying the free monoidal category functor to the equalizer of \((f',g')\) does not yield an equalizer in general, but there still exists a canonical map \(\clB^* \to \text{eq}((f')^*,(g')^*) = \clE\); this yields a unique way to fill the triangle
	\[\vxy[@C=1.3cm]{
		\ar[d]\clA\times\opid\clQ\times (\clB^*)^\op \ar[dr]^-{s} \\
		\ar[d]\clA\times\opid\clQ\times\clE &\Set \\
		\clA_0\times\opid{\clQ_0}\times (\clB^*_0)^\op \ar[ur]_-{s_0}
		}\]
	with an invertible natural transformation. The verification that this yields the equalizer of \eqref{eq_cells} is straightforward.
\end{construction}
Cotabulators are also not difficult to construct mimicking their construction in \(\DProf\); on the other hand, few tabulators exist: the problem seems to be that cells of shape
\[\vxy{
	\clX \ar@{}[dr]|{\sm[s]{q_0 & \alpha}}\ar@{=}[r]|-@{*}\ar[d] & \clX\ar[d] \\
	\clA \ar[r]|-@{*}_{(\clQ,t)} & \clB
	}\]
with identity loose domain specify, among other data, a pointing for the state category \(\clQ\); a limit of this kind exists only if \(\clQ\) can be `universally pointed', and this is the case only if \(\clQ\) is a category with a single object (and morphism, imposing the universal property in both dimensions). Let's expand on both of these points.
\begin{construction}[Cotabulators in \(\DTDX\)]\label{cotabs}
	The construction of cotabulators goes as in the double category of profunctors. Let's recall what the universal property of \(\cotab(\clQ,t)\) is: there must be a cell (a bit more conveniently depicted as triangular, collapsing the horizontal side if it's a loose identity)
	\[\vxy{
		\clA \ar[dr]\ar[rr]^-{(\clQ,t)}&\dtwocell<\omit>{\sm[r]{!&\kappa}}& \clB\ar[dl]\\
		&\cotab(\clQ,t)&
		}\]
	such that every	other cell as in the left below factors uniquely as in the right:
	\[\label{ahfdsoaf}\vxy[@R=6mm]{
		\clA \ar[ddr]\ar[rr]^-{(\clQ,t)}&\dtwocell<\omit>{\sm[r]{!&\theta}}& \clB\ar[ddl] & \clA\ar[rr]^-{(\clQ,t)}\ar[dr] &\dtwocell<\omit>{\sm[r]{!&\kappa}}& \clB\ar[dl]\\
		& & & &\cotab(\clQ,t)\ar[d]_Y& \\
		&\clY& &  &\clY&
		}\]
	In this construction, \(\cotab(\clQ,t)\) is a category \(\clK\) with its hom functor and \(\kappa\) is a natural transformation of type
	\[\vxy{t(a,q,q')(\unb) \ar[r] & \clK(j_\clA a,j_{\clB^*}\unb)}\]
	so that a natural guess for \(\kappa\) is the natural transformation obtained for first taking the profunctor \(t_0:=\colim_{\opid\clQ}t : \clA\pto\clB^*\), and then the collage \(\clA\uplus_{t_0}\clB^*\), equipped with the well-known canonical cospan \(\clA \xto{j_\clA}\clA\uplus_{t_0}\clB^*\xot{j_{\clB^*}}\clB^*\). This is easily seen to yield the universal property just stated, since a natural transformation \(\theta\) as in \eqref{ahfdsoaf} on the left induces a unique cell in the double category \(\DProf\) of shape
	\[\vxy{
		\clA \ar[d]_F \ar[r]|-@{*}^{t_0}& \clB^*\ar[d]^{G^*} \\
		\clY \ar[r]|-@{*}_h& \clY^*
		}\]
	whence a unique cell in \(\DTDX\) as on the right.
\end{construction}
\begin{construction}[Obstruction to the existence of tabulators in \(\DTDX\)]\label{not_all_tabs}
	Not all tabulators exist	in \(\DTDX\); for a 1-cell \((\clQ,t) : \clA\tdto\clB\) to admit a tabulator there must be a cell of the form
	\[\vxy{
		&\tab(\clQ,t)\dtwocell<\omit>{\sm[r]{q_0&\tau}}\ar[dr]^{\tab_r}\ar[dl]_{\tab_l}&\\
		\clA \ar[rr]|-@{*}_-{(\clQ,t)}&& \clB
		}\]
	such that every other cell as in the left below factors uniquely as in the right:
	\[\label{facto_tabu_diags}\vxy[@R=6mm]{
		& \clX\ddtwocell<\omit>{\sm[r]{q&\chi}}\ar[ddr]^R\ar[ddl]_L && & \clX\ar[d]^X\\
		&& & &\tab(\clQ,t)\dtwocell<\omit>{\sm[r]{q_0&\tau}}\ar[dr]^{\tab_r}\ar[dl]_{\tab_l}& \\
		\clA \ar[rr]|-@{*}_{(\clQ,t)}&& \clB & \clA \ar[rr]|-@{*}_{(\clQ,t)}&& \clB
		}\]
	The cell \(\sm[r]{q&\chi}\) chooses an object of \(\clQ\); but in order for the commutativity above to hold, this object has to be equal to \(q_0\); if \(\clQ\) does not have a single object, this can't be done. Given this, the tabulator of a transducer \((M,t)\) having a monoid as state category, of the form
	% then exist in \(\DTDX\) only if \(\clQ\) is a category with a single object (a monoid), and in that case they are computed as in \([\clA\times(\clB^*)^\op,\Set]\emdash\DProf\), in the sense specified in \autoref{has_tab}: given a monoid \(M\), one can regard a transducer \((M,t) : \clA\tdto\clB\) as a functor of type 
	\[\vxy{
			t : \clA\times\opid M\times (\clB^*)^\op \ar[r] & \Set
		}\]
	is computed regarding the transducer as a \(\Set\)-profunctor of type \(\clA\times M \pto M \times\clB^*\); explicitly, the tabulator of \(t\) is the category \(\tab(M,t)\) having
	\begin{itemize}
		\item objects the triples \((a,\unb;\xi)\) where \(\xi  :  t(a,\bullet,\bullet)[\unb]\) (\(\bullet\) is the unique object of the category \(M\));
		\item arrows \((a,\unb;\xi)\to(x,\uny;\zeta)\) the quadruples \(\sm[r]{u	&	v \\ m	&	n}\) such that \(u : a\to x\), \(v : b\to y\), and \(m,n :  M\) are monoid elements, so that the two functions
		      \[\vxy[@R=-1mm@C=2cm]{
			      {t(a,\bullet,\bullet)[\unb]} \ar[r]^-{t(u,\bullet,n)[\unb]} & t(a,\bullet,\bullet)[\uny] & \ar[l]_-{t(a,m,\bullet)[v^*]}t(x,\bullet,\bullet)[\uny]\\
			      \xi \ar@{|->}[r]& \# & \ar@{|->}[l]\zeta
			      }\]
		      map \(\xi,\zeta\) to the same element, \ie \(t(u,\bullet,n)[\unb](\zeta) = t(a,m,\bullet)[v^*](\xi)\); the natural transformation \(\tau\) filling the tabulator cell sends an arrow as above into such common value in \(t(a,\bullet,\bullet)[\uny] =t(\tab_l(a,\unb;\xi),\bullet,\bullet)[\tab_r^*(x,\uny;\zeta)]\).
	\end{itemize}
	With this definition, one rotely proves a terminality property of \(\tau\) in the shape of \eqref{facto_tabu_diags}; a span \(\clA \xot L\clX\xto R\clB\) must define a unique functor \(\clX \to \tab(M,t)\) suitably applying \(L,R\).
\end{construction}
\begin{construction}[Companions and conjoints of tight arrows in \(\DTDX\)]
	It's easy to observe that the embedding \(\DProf\subseteq\DTDX : p\mapsto\bar p\) of \autoref{natural_embed} sends companions and conjoints of tight arrows \(F :\clA\to\clB\) to companions and conjoints of \(F\) regarded as a tight arrow in \(\DTDX\), meaning that the companion \(F_{*,\DTDX} =: F_\natural\) of \(F\) is \(\overline{(F_{*,\DProf})}\), and similarly for the conjoint.

	The construction of cells
	\[\vxy{
		\clA \ar@{=}[r]|-@{*} \ar@{=}[d] \ar@{}[dr]|{\sm[r]{\bang & \eta_\natural}} & \clA \ar@{->}[d]^{F} & \clA \ar@{->}[r]|-@{*}^{F_\natural} \ar@{->}[d]_{F} \ar@{}[dr]|{\sm[r]{\bang & \epsilon_\natural}} & \clB \ar@{=}[d] & \clB \ar@{->}[r]|-@{*}^{F^\natural} \ar@{=}[d] \ar@{}[dr]|{\sm[r]{\bang & \epsilon^\natural}} & \clA \ar@{->}[d]^{F} & \clA \ar@{->}[d]_{F} \ar@{=}[r]|-@{*} \ar@{}[dr]|{\sm[r]{\bang & \eta^\natural}} & \clA \ar@{=}[d] \\
		\clA \ar@{->}[r]|-@{*}_{F_\natural} & \clB & \clB \ar@{=}[r]|-@{*} & \clB & \clB \ar@{=}[r]|-@{*} & \clB & \clB \ar@{->}[r]|-@{*}_{F^\natural} & \clA
		}\]
	related by the companion and conjoint identities is straightforward, when one defines
	\begin{itemize}
		\item \(F_\natural\) as the 2-transducer \((\clI,F_\natural) : \clA\times\opid\clI\times (\clB^*)^\op \to \Set\) sending \((a,\bullet,\bullet,\unb)\) to \(\clB(F b,a)\) if \(\unb = b\kons \emptyList\) and \(\varnothing\);
		\item \(F^\natural\) as the 2-transducer \((\clI,F^\natural) : \clB\times\opid\clI\times (\clA^*)^\op \to \Set\) sending \((b,\bullet,\bullet,\una)\) to \(\clB(a,Fb)\) if \(\una = a\kons \emptyList\) and \(\varnothing\).
	\end{itemize}
\end{construction}
\subsection{Monoidality}\label{sec_monoidality}
The bicategory \(\TwoTDX\) is probably a monoidal bicategory in the sense of \cite{coherence-tricat}, but proving all necessary coherences is a daunting task. A more sleek approach is prescribed by \cite{hansen2019constructingsymmetricmonoidalbicategories}, and leverages on the existence of conjoints; we are content with recording, instead, the same result for the decategorified version of \(\TwoTDX\), the posetal bicategory of 1-transducers
\[\vxy{t : A^*\times Q \times Q \ar[r] & \pow{B^*}}\]
and transducer morphisms \((s,P)\to (t,Q)\) of \autoref{bic_of_1_tnd}.

The claim is that the (locally posetal) bicategory \(\TDX\) so defined is a monoidal bicategory (actually, a strict monoidal 2-category): at the best of our knowledge, even at this posetal level, this result is novel and it avoids coherence issues completely, since in a poset all diagrams commute. %So, in the following we will show that \(\TDX\) is a strict monoidal bicategory, but w
We will however state the definitions in the language and notation of \(\TwoTDX\).
\begin{remark}\label{stren}
	The functor \(\Set^{(\firstblank)^\op}\) sending a category to its presheaf category comes equipped with a tensorial strength having components
	\[\vxy{\tau : \freerig \clX\times \freerig \clY \ar[r] & \Set^{(\clX^*\times \clY^*)^\op} }\]
	defined as \((F,G)\mapsto \big((x,y)\mapsto Fx\otimes Gy\big)\).

	Composed with the functor \(\Pi=\langle\pi_\clX^*,\pi_\clY^*\rangle : (\clX\times\clY)^*\to \clX^*\times\clY^*\), this yields a functor
	\[\vxy{\sigma : \freerig \clX\times \freerig \clY \ar[r]^-\tau & \Set^{(\clX^*\times \clY^*)^\op} \ar[r]^-{\firstblank\circ\Pi}& \freerig{(\clX\times\clY)}}\]
\end{remark}
In order to exhibit a monoidal bicategory, we have to provide:
\begin{itemize}
	\item a 2-functor
	      \[\vxy{\firstblank\otimes\firstblank : \TwoTDX \times \TwoTDX \ar[r] & \TwoTDX.}\]
	      This is easy to define on objects, as \((\clA,\clC)\mapsto\clA \otimes_0\clC\) coincides with the product of categories \(\clA\times\clC\), and also on 1-cells, where one exploits the map
	      \(\sigma\) of \autoref{stren}. More precisely, let \((\clQ,s) : \clA\tdto\clB\) and \((\clP,t) : \clC\tdto\clD\) be two \twoDash transducers; consider the composite map \((\clQ,s)\otimes_1(\clP,t)\)
	      \[\vxy{
		      \clA^* \times  \clC^*\times ( \clP\times \clQ)\times ( \clP\times \clQ)\ar@{=}[d]_\wr \ar[r]& \freerig{( \clB\times \clD)}\\
		      (\clA^* \times  \clP\times \clP) \times  (\clC^*\times  \clQ\times \clQ) \ar[r]_-{t\times s}& \freerig \clB \times\freerig \clD\ar[u]_-\sigma
		      }\]
	      this defines \(\otimes\) on 1-cells.

	      Finally, given two 2-cells \((F,\theta) : s\To t\) and \((G,\omega) : s'\To t'\) (in the posetal case, this is the part that trivializes, as \(\theta\) and \(\omega\) reduce to pointwise inequalities in the free quantales \(\freerig B\) and \(\freerig D\) respectively),
	      \[\vxy[@R=5mm]{
		      \clA\times \opid{\clP} \drtwocell<\omit>{<3>\theta}\ar[dd]_-{\clA\times F^\op\times F}\ar[dr]^-s & &&
		      \clC\times \opid{\clX} \drtwocell<\omit>{<3>\omega}\ar[dd]_-{\clC\times G^\op\times G}\ar[dr]^-{s'} &\\
		      & \freerig{\clB} &&
		      &\freerig{\clD}\\
		      \clA\times \opid{\clQ} \ar[ur]_-{t}& &&
		      \clC\times \opid{\clY} \ar[ur]_-{t'}&
		      }\]
	      one defines \((F,\theta)\otimes (G,\omega) : s\otimes t \To s'\otimes t'\) as the 2-cell arising from the pasting
	      \[\vxy[@R=5mm]{
			      \clA\times\opid{\clP}\times\clC\times\opid{\clX}
			      \drtwocell<\omit>{<3>\kern1.5em\theta\times\omega}
			      \ar@{->}[dd] \ar@{->}[rrd] &  &  &  &  \\
			      &  & \freerig\clB\times\freerig\clD \ar@{->}[rr]^{\sigma} &  & \freerig{(\clB\times\clD)} \\
			      \clA\times\opid{\clQ}\times\clC\times\opid{\clY} \ar@{->}[rru] &  &  &  &
		      }\]
	      Observe how in terms of the double cells of \eqref{dbl_cell} this is simply the pointwise product
	      \[\vxy{
		      \clP\times\clX\drtwocell<\omit>{ }\ar[r]|-@{|}^{(s\times t)(a,-)}\ar[d]_-{F\times G} & \clP\times\clX\ar[d]^-{F\times G} \\
		      \clQ\times\clY\ar[r]|-@{|}_{(s'\times t')(a,-)} & \clQ\times\clY.
		      }\]
	      This concludes the construction of the 2-functor \(\otimes\).
	\item A two-sided unit for the \(\otimes\) operation; because of the way the tensor is defined, the only reasonable candidate for such a unit \(i : I\tdto I\) is the transducer \((\{0\},i)\) where \(\{0\}\) is a singleton state space, \(I = \{\bullet\}\) another singleton set, and
	      \[\vxy{i : I^* \times \{0\}\times \{0\}\ar[r] & \freerig{I}}\]
	      is the map sending \((n,0,0) :  I^*\times \{0\}\times \{0\} = \bbN\times \{0\}\times \{0\}\) to the singleton \(\{n\}\).
\end{itemize}
The invertible modifications yielding associators and unitors are, in the posetal case, equalities, as well as the pentagonator and the two 2-unitors, hence we get a strict monoidal 2-category \((\TDX,\otimes,I)\).
\subsection{Adjunctions and monads}\label{sec_adjunctions-monads}
Let's start spelling out the	definition of a monad in \(\DTDX\); we will find something similar to what happens with a monad in the bicategory \(\Prof\) of profunctors, which is well-known to be a category structure; here, a grading induced by the category of states enters the picture.

Adjunctions are the following notion awaiting to be analyzed, but they will turn out to be not very interesting (\cf \autoref{very_few} below; from this we derive that adjunctions can generate only rather uninteresting monads).
\begin{definition}[Monad in \(\DTDX\)]\label{whatsa_monad}
	A monad \(\Big(\clA,(\clQ,t),(\jmath,\eta),(\boxtimes,\mu)\Big)\) in \(\DTDX\) consists of
	\begin{itemize}
		\item a small category \(\clA\);
		\item a \twoDash transducer \((\clQ,t) : \clA \tdto \clA\), \ie a functor of type
		      \[\vxy{t : \clA^* \times \opid{\clQ} \times (\clA^*)^\op \ar[r] & \Set}\]
		\item equipped with a \emph{unit} 2-cell \((\jmath,\eta) : i\To t\) in \(\TwoTDX(\clA,\clA)\), where \(i\) is the identity transducer on \(\clA\); this is realized as a natural transformation
		      \[\label{etas}\vxy[@R=5mm]{
			      \clA\times \opid{\clI} \drtwocell<\omit>{<3>\eta}
			      \ar[dd]_-{\clA\times \jmath^\op\times \jmath}\ar[dr]^-i & \\
			      & \freerig{\clA} \\
			      \clA\times \opid{\clQ} \ar[ur]_-{t}&
			      }\]
		      Here \(\jmath : \clI\to\clQ\) picks an object of \(\clQ\) and (after extending \(t\) to \(\clA^*\)) \(\eta : \clA^*(\firstblank,\una)\To t(\una,\jmath,\jmath)\) is a natural transformation of presheaves on \(\clA^*\) (by Yoneda, this picks a distinguished element in \(t(\una,\jmath,\jmath,\una)\) for every \(\una : \clA^*\)),
		      and
		\item equipped with a \emph{multiplication} 2-cell \((\boxtimes,\mu):t\tranComp t \To t\), realized as a natural transformation
		      \[\label{mus}\vxy[@R=5mm]{
			      \clA\times (\clQ\times\clQ)^\op\times(\clQ\times\clQ) \drtwocell<\omit>{<3>\mu}
			      \ar[dd]_-{\clA\times \boxtimes^\op\times \boxtimes}\ar[dr]^-{t\tranComp t} & \\
			      & \freerig{\clB} \\
			      \clA\times \opid{\clQ} \ar[ur]_-{t}&
			      }\]
		      where \(\firstblank\boxtimes\firstblank : \clQ\times\clQ \to\clQ\) is a functor of two arguments, and \(\mu : t\tranComp t \To (t\cdot (\clA\times \boxtimes^\op\times \boxtimes))\) is a natural transformation with components (each of which is a natural transformation)
		      \[\vxy{\mu_{(\unu;(x,x'),(y,y'))} : \displaystyle\int^{\una} t(\unu;x,y)[\una]\times t(\una, x',y') \ar@{=>}[r] & t(\unu;x\boxtimes x',y\boxtimes y').}\]
		\item These data are subject to unitality and associativity axioms, expressed by the commutativities of suitable diagrams (which we will display once we set up a slightly better notation, below).
	\end{itemize}
\end{definition}
The corollary that we obtain is easily stated as follows: the state category \(\clQ\) part of the data that a monad in \(\DTDX\) specifies, is equipped by \((\boxtimes,\jmath)\) with a monoidal structure. Furthermore, each \(t(-,q,q',-) : \clA^* \times(\clA^*)^\op\to\Set\) defines a category structure on \(\clA^*\).
The maps in \eqref{mus} in fact are given in components at the object \(\unv\) by functions of type
\[\vxy{\mu_{(\unu;(x,x'),(y,y'))} : \displaystyle\int^{\una} t(\unu;x,y)[\una]\otimes t(\una, x',y')[\unv] \ar[r] & t(\unu;x\boxtimes x',y\boxtimes y')[\unv]}\notag\]
and if now we denote \(t(\unu;x,y)[\una]\) as \(t(\unu,\una)_{xy}\), then \(\mu\) corresponds to a cowedge in \(\una\) of type
\[\label{graded_composizia}\vxy{ t(\unu,\una)_{xy}\otimes t(\una,\unv)_{x',y'} \ar[r] & t(\unu,\unv)_{x\boxtimes x',y\boxtimes y'}} \]
subject to the condition that the diagram
\[\vxy[@R=4mm]{
	t(\unu,\una)_{xy}\otimes\Big(t(\una,\unv)_{x'y'}\otimes t(\unv,\unw)_{x''y''}\Big) \ar@{=}[d]_-\wr\ar[rr]^-{1\otimes \mu_{\una\unw,(x',y')(x'',y'')}^{\unv}} &&t(\unu,\una)_{xy}\otimes t(\una,\unw)_{x'x'',y'y''} \ar[dd]^{\mu_{\unu\unw,(x,y)(x'x'',y'y'')}^{\una}}\\
	\Big(t(\unu,\una)_{xy}\otimes t(\una,\unv)_{x'y'}\Big)\otimes t(\unv,\unw)_{x''y''} \ar[dd]_{\mu_{\unu\unv,(x,y)(x',y')}^{\una}\otimes 1}&&\\
	&&t(\unu,\unw)_{x(x'x''),y(y'y'')}\ar@{=}[d]^-\wr\\
	t(\unu,\unv)_{xx',yy'}\otimes t(\unv,\unw)_{x''y''}\ar[rr]_{\mu_{\unu\unw,(xx',yy')(x'',y'')}^{\unv}}&&t(\unu,\unw)_{(xx')x'',(yy')y''}
	}\]
commutes; a similar diagram expresses the unit law asserting that the pasting of 2-cells \(\mu\tranComp (t * \eta)\) obtained as
\[\vxy[@R=7mm]{
	\clA^*\times\opid{(\clQ\times\clI)} \drtwocell<\omit>{<3>\kern-5.5em\eta\times\opid\clQ}\ddrrtwocell<\omit>{<8>\mu}\ar@{->}[dd]_{\clA^*\times\opid{(\clQ\times \jmath)}} \ar@{->}[rd] \ar@/^1pc/@{->}[rrd]^-t &  &  \\
	& \freerig\clA\times\opid\clQ \ar@{->}[r]_-T & \freerig\clA \\
	\clA^*\times\opid{(\clQ\times\clQ)} \ar@{->}[d]_{\clA^*\times\opid\boxtimes} \ar@{->}[ru] &  &  \\
	\clA^*\times\opid\clQ \ar@/_2pc/@{->}[rruu]_-t &  &
	}\]
is the identity 2-cell of \(t\) (and in turn, a similar diagram expresses the other unit law \(\mu\tranComp (\eta * t) = \id_t\)).

\medskip
Packaging all together (perhaps unsurprisingly), we obtain the following result.
\begin{theorem}
	\label{monad_in_TwoTDX}
	A monad in \(\DTDX\) consists of a promonad on the category \(\clA^*\), so that each \(t(\una,q,q',\una')\) specifies a set of \emph{heteromorphisms}, compatible with the morphisms in \(\clA^*\) (=finite tuples of morphisms in \(\clA\)); furthermore, this promonad/category is \emph{graded} over the state category \(\opid{\clQ}\) (equipped with the pointwise monoidal structure) in the sense that composition of something `of degree \((x,y)\)', together with something `of degree \((x',y')\)' is of degree \((x\boxtimes x',y\boxtimes y')\), in the sense specified by \eqref{graded_composizia}.
\end{theorem}
Regarded as a functor \(\clA^*\to\freerig\clA\emdash\Prof(\clQ,\clQ)\), the transducer \(t\) is strong monoidal. This means, essentially, that the free monoidal structure on \(\clA^*\) is compatible with the heteromorphisms added by the promonad \(t\).
\subsubsection{Adjunctions}
Characterization \ref{tdx_char_6} of \(\TwoTDX(\clA,\clB)\) in \autoref{remark_many_versions_of_a_category} allows for the definition of adjunctions in \(\TwoTDX\) as a special case of the definition of adjunctions in a category of a specific form; the following definition comes from \cite[Example 5]{Gaboardi2021}, of which this is the (2-categorical) verbatim dual. Recall that a \emph{\(\clM\)-graded comonad} on a category \(\clC\), consists of an oplax functor \(K : \Sigma\clM \to \Cat\), if \(\clM\) is a monoidal category regarded as a single-object bicategory. The image under \(K\) of the unique object in \(\Sigma\clM\) selects a category \(\clC\), and to every object of \(\clM\) is associated a functor \(K(M,-) : \clC\to\clC\), satisfying appropriate conditions that express its colaxity.  % Clearly, \(K\) consists of an oplax monoidal functor \(K : (\clM,\times)\to(\Cat,\times)\).
To each such comonad one can associate a `locally graded' version of the coKleisli construction.

\begin{definition}
	\label{locally_graded_coKleisli}
	The \emph{coKleisli \(\clM\)-graded category} \(\coKl^\clM(K)\) of a graded comonad \(K\) as above has
	\begin{itemize}
		\item objects the same of \(\clC\), denoted \(X,Y,Z,\dots\);
		\item morphisms \(X\pto Y\) the pairs \((M,\phi)\) where \(M : \clM\) is an object and \(\phi : K(M,X)\to Y\) is a morphism in \(\clC\);
		\item composition of \((M,\phi) : X\pto Y\) with \((N,\psi) : Y\pto Z\) given by the formula
		      \[\vxy{
			      K(N\otimes M,X) \ar[r] & K(N,K(M,X)) \ar[r]^-{KN\phi} & K(N,Y) \ar[r]^-\psi & Z
			      }\]
		\item identity morphisms are given by part of the oplax structure of \(K\), \ie by the unitor \(\epsilon_X : K(1,X)\to X\).
	\end{itemize}
	Using the oplax structure of \(K\) one can prove that composition is associative and unital.
\end{definition}
\begin{remark}
	In our case of interest, the graded comonad is the 2-functor \(K_\clQ\) sending \(\clC\) to \(\clC\mapsto\clC\times\opid{\clQ}\), which can be seen as restriction
	\[\vxy[@R=-1mm]{
			\Cat\times\Cat \ar[r] & \Cat\times\Cat \ar[r] & \Cat\\
			(\clQ,\clC) \ar@{|->}[r] & (\opid\clQ,\clC) \ar@{|->}[r] & \clC\times\opid{\clQ}
		}\]
	of the `writer comonad' functor \(\clX\times-\) to categories of type \(\opid\clQ\).\footnote{The operation \(\clQ\mapsto \opid\clQ\) is itself a comonad, the `opid' comonad, whose existence is exploited --in a restricted form apt to capture dagger categories as coalgebras-- in \cite[2.1.10]{karvonen2019waydagger}.} From which we obtain the locally graded coKleisli category having
	\begin{itemize}
		\item objects the categories \(\clA,\clB,\clC\), etc.;
		\item morphisms \((\clQ,f) : \clA\pto\clB\) the pairs where \(\clQ\) is a category	and \(f : \clA\times\opid\clQ\to\clB\) is a functor;
		\item composition of \(\clA\mathrel{\overset{(\clQ,f)}\pto}\clB\mathrel{\overset{(\clP,g)}\pto}\clC\) given by \((\clQ\times\clP,g\bullet f)\) where	\(g\bullet f\) is the composite
		      \[\vxy[@R=3mm@C=1.5cm]{
			      \clA\times\opid{(\clQ\times\clP)}\ar@{=}[d]\\
			      (\clA\times\opid\clQ)\times(\opid\clP) \ar[r]^-{f\times\opid\clP}& \clB\times\opid\clP \ar[dd]^g\\ & \\ & \clC
			      }\]
		\item identity given by the unitor \((\clI,\lambda_\clA : \clA\times\clI \to\clA)\).
	\end{itemize}
	This can evidently be identified with the category \(\TwoTDX(\clA,\clB)\) of \ref{tdx_char_6} in \autoref{remark_many_versions_of_a_category}.
\end{remark}
Now, we need a technical lemma that says that there are very few adjunctions in the coKleisli graded 2-category of the writer comonad.
\begin{lemma}
	\label{adjunction_in_coKl}
	Let \((\clX,f) : \clX\times\clA\to\clB\) and \((\clY,g) : \clY\times\clB\to\clA\) be two adjoint 1-cells in the graded coKleisli 2-category \(\coKl^\Cat(W)\), if \(W(\clX,\clA):=\clX\times\clA\); then it is necessary that \(\clX=\clY=\clI=\{\bullet\}\) be the terminal category. Ultimately, an adjunction in \(\coKl^\Cat(W)\) boils down to a single adjunction \(f_\bullet:\clA\leftrightarrows\clB:g_\bullet\) in \(\Cat\), associated to the only object of \(\clI\).
\end{lemma}
\begin{proof}
	Uncurrying \(f\) and \(g\) we obtain maps
	\[\vxy{
			\clX\ar[r]^-f & [\clA,\clB] & f : \clY\ar[r]^-g & [\clB,\clA]
		}\]
	each of which induces a functor \(f_x : \clA\to\clB\)	and \(g_y : \clB\to\clA\) for every object \(x : \clX\) and \(y : \clY\). Unit and counit	of the adjunction are given by the natural transformations filling the diagrams
	\[\vxy[@R=5mm]{
		\clX\times\clY \drrtwocell<\omit>{<8>\epsilon}\ar@{->}[rd]^{f\times g} \ar@{->}[ddd]_{\boldsymbol !} &  &  & \clI \drrtwocell<\omit>{<5>\eta}\ar@{->}[rrd]^-{[\id_\clA]} \ar@{->}[ddd]_{(y_0,x_0)} &  &  \\
		& [\clA,\clB]\times[\clB,\clA] \ar@{->}[rd]^-\circ &  &  &  &  [\clA,\clA] \\
		&  & [\clB,\clB]&  & [\clB,\clA]\times[\clA,\clB] \ar@{->}[ru]_-\circ &  \\
		\clI \ar@{->}[rru]_-{[\id_\clB]} &  &  & \clY\times\clX \ar@{->}[ru]_-{g\times f} &  &
		}\]
	so that there is a natural transformation
	\[\vxy{\eta_a^{y_0x_0} : a \ar[r] & g_{y_0}(f_{x_0}a)}\]
	associated to the parameters \((y_0,x_0) : \clY\times\clX\) and natural in \(a :  \clA\), and a counit
	\[\vxy{\epsilon_b^{xy} : f_x(g_yb) \ar[r]& b}\]
	which is a cocone in \((x,y) : \clX\times\clY\) and thus induces a unique \(\bar\epsilon	: \colim_{xy}f_x(g_yb)\to b\). Now, it's easy to see that the triangle identities assert that
	\begin{itemize}
		\item the composition \(\xymatrix{ f_xa \ar[r] & f_x(g_{y_0}(f_{x_0}a)) \ar[r] & f_{x_0}a }\) is equal to the identity of \(f_xa\), and the composition
		      \[\vxy[@R=0cm]{
				      \clX \ar[r] & \clX\times(\clY\times\clX)  \ar[r] & \clX\\
				      x\ar@{|->}[r]\ar@(ul,dl) & (x,y_0,x_0) \ar@{|->}[r] & x_0\ar@(ur,dr)^{\id}
			      }\]
		      must be the identity of \(\clX\) (recall that a 2-cell is a pair made of a functor and a natural transformation filling a triangle); thus there is only one object \(x_0\) in \(\clX\), and only one identity arrow.
		\item Similarly, the composition \(\xymatrix{ g_yb \ar[r] & g_{y_0}(f_{x_0}(g_yb)) \ar[r] & g_{y_0}b }\) is equal to the identity of \(g_{y_0}b\), and the composition
		      \[\vxy[@R=0cm]{
				      \clY \ar[r] & (\clY\times \clX)\times\clY  \ar[r] & \clY\\
				      y\ar@{|->}[r]\ar@(ul,dl) & (x_0,y_0,y) \ar@{|->}[r] & y_0\ar@(ur,dr)^{\id}
			      }\]
		      must be the identity of \(\clY\).
	\end{itemize}
	But then, the same triangle identities imply that there is an adjunction \(f_{x_0}\dashv g_{y_0}\), with unit \(\eta^{x_0y_0}\) and counit the map \(\epsilon_0 : f_{x_0}(g_{y_0}b)\xto{\text{in}_{(x_0,y_0)}} \colim_{xy}f_x(g_yb)\to b\). This concludes the proof	of the lemma.
\end{proof}
\begin{corollary}\label{very_few}
	Adjunctions in \(\TwoTDX\) are very few. Unpacking the definition of adjunction valid in any 2-category, a pair of adjoint transducers consists of
	\begin{itemize}
		\item the left adjoint transducer \((\clX,f) : \clA\tdto\clB\), and the right adjoint transducer \((\clY,g) : \clB\tdto\clA\);
		\item the unit 2-cell
		      \[\vxy{(H,\eta) : (\clI,\iota_\clA) \ar@{=>}[r] & (\clY\times\clX,g\circ f)}\]
		      consisting of a functor \(H:\clI\to\clX\times\clY\) and a natural transformation \(\eta\), having the effect of picking an object \((x_0,y_0) :  (\clX\times\clY)_0\), and yielding maps (natural in \(\una,\una'\))
		      \[\eta_{aa'} : \clA^*(\una',\una) \to G(f(\una,x_0,x_0),y_0,y_0)[\una'].\]
	\end{itemize}
	The unit map, together with the counit, provide an adjunction \(f_{x_0x_0} = f(-,x_0,x_0)\dashv g(-,y_0,y_0)=g_{y_0y_0}\).
\end{corollary}

\section{\(\TwoTDX\) and other bicategories of automata}\label{sec_relations-automata}
The scope of this section is to compare \(\TwoTDX\) with other `bicategories of processes', most importantly the bicategory of \emph{circuits} studied by Walters et al.\@\xspace in \cite{Katis1997} and the bicategory \(\MAC\) of \emph{machines} described by Guitart \cite{guitart1974remarques}.
\subsection{\(\TwoTDX\) and bicategories of circuits}
We start by recalling the definition of the bicategory \(\Mly(\clK)\) of \emph{Mealy automata} valued in a Cartesian category \(\clK\).\footnote{We will be particularly interested in two cases: \(\clK\) equal to the category of sets and functions, and the 1-category of categories and functors; generalizing this picture is possible (the special case of a topos as Cartesian closed category was studied in \cite{hora2024topoiautomataitopoi}, and \cite{guitart1974remarques} takes into account some features of \(\Cat\) qua 2-category, not just 1-category.) but out of the scope of the present section, \ie building a comparison functor \(\Mly(\clK) \to\TDX\) and \(\Mly(\Cat)\to\TwoTDX\).} The construction that follows comes from \cite{Katis2010,Katis1997}.

Let \(\bbN\) denote the one-object, 2-discrete bicategory on the graph with one vertex and one edge (it's `the monoidal category of natural numbers': hence the name). For every bicategory \(\clB\), define \(\Omega\clB\) as the bicategory of pseudofunctors \(A : \bbN\to \clB\), lax natural transformations between such functors, and modifications between such lax natural transformations. Let \(\Sigma\clK\) denote the monoidal category \((\clK,\times,1)\), regarded as a single-object bicategory.
\begin{definition}[The bicategory of Mealy automata]\label{def:mly}
	The \emph{bicategory of Mealy automata}, denoted as \(\Mly(\clK)\), is defined as \(\Omega(\Sigma\clK)\).%, where \(\Sigma : \MonCat \to\BiCat\) is the functor that regards a monoidal category  as a one-object bicategory.
\end{definition}
Explicitly,
\begin{enumtag}{ml}
	\item\label{ml_1} an object of \(\Mly(\clK)\) consists of a mapping sending the unique object of \(\bbN\) to the unique object \(\bullet\) of \(\Sigma\clK\); the unique morphism generator of \(\bbN\) then picks an object \(A : \clK\), together with all its iterates \(A^{\times n}\) (image of \(n : \bbN\)).
	\item\label{ml_2} A 1-cell of \(\Mly(\clK)\) is a \emph{Mealy automaton}, \ie a tuple \((E,d,s)\) where \(E : \clK\) is an object of `states' (witnessing the unique component of a lax natural transformation between two pseudofunctors \(A,B : \bbN \to\Sigma\clK\)), and a laxity cell \(\langle d,s\rangle : A\times X \to X\times B\), or rather a span
	\[\label{generic_mealy}\vxy{X & A\times X \ar[r]^-s \ar[l]_-d & B.}\]
	Note how \(\Omega(\Sigma(\clK,\otimes,I))\) makes sense for every monoidal category, but this span representation exists only if \(\clK\) is Cartesian.
	\item\label{ml_3} A 2-cell \(f : (X,d,s)\To (Y,d',s')\), \ie a modification, consists of a map between state spaces, \(f : X\to Y\) such that the following diagram commutes:
	\[\label{generic_mealy_map}\vxy{X \ar[d]_f & A\times X \ar[r]^-s \ar[l]_-d\ar[d]^{A\times f} & B\ar@{=}[d]\\
		Y & A\times Y \ar[r]^-{s'} \ar[l]_-{d'} & B.}\]
\end{enumtag}
Composition of spans like
\[\vxy{X & A\times X \ar[r]^-s \ar[l]_-d & B & B\times Y \ar[r]^-{s'} \ar[l]_-{d'}& C}\]
consists of the pasting of lax natural transformations, which when destructured appears as a composite span
\[\vxy{X \times Y & A\times (X \times Y)\ar[r]^-{s'\diamond s} \ar[l]_-{d'\diamond d} & C}\]
where \(d'\diamond d\) and \(s'\diamond s\) are given respectively by (given in the category of sets for convenience, but it's easy to lambda-abstract the definition to work in any Cartesian \(\clK\))
\[\begin{cases}
		d'\diamond d (a, (x,y)) = \big(d(\una,x), d'(s(\una,x),y)\big), \\
		s'\diamond s (a, (x,y)) = s'(s(\una,x),y).
	\end{cases}\]
\begin{remark}\label{remark_mly_universal_properties}
	Fixed two objects \(A,B : \clK\), the hom-category \(\Mly(A,B)\) enjoys several universal properties:
	\begin{enumtag}{mp}
		\item \label{mly_up_1} if \(\clK\) is monoidal closed, the hom-category \(\Mly(A,B)\) is isomorphic the category of coalgebras for the endofunctor
		\[\vxy[@R=0cm]{
				R_{AB}:\clK \ar[r]& \clK\\
				X\ar@{|->}[r] & [A,B\times X].
			}\]
		As a consequence, there exists an inserter diagram
		\[\vxy{
			\Mly(A,B) \drtwocell<\omit>{}\ar[r]\ar[d] & \clK \ar@{=}[d]\\
			\clK \ar[r]_{R_{AB}}& \clK.
			}\]
		\item \label{mly_up_2} the hom-category \(\Mly(A,B)\) fits in the strict 2-pullback
		\[\vxy{
			\Mly(A,B) \xpb\ar[r]\ar[d] & (A\times\blank/B) \ar[d]^U\\
			\Alg(A\times\blank) \ar[r]_-{U'}& \clK,
			}\]
		where \(\Alg(A\times\blank)\) is the category of endofunctor algebras for \(A\times- : \clK\to\clK\), and \((A\times-/B)\) the comma category of arrows \(A\times X\to B\) and \(U,U'\) are forgetful functors.
	\end{enumtag}
	Similarly to what happens with the numerous characterizations of \autoref{remark_many_versions_of_a_category}, it is sometimes more convenient, depending on the kind of problem at hand, to argue with the bare definition of \(\Mly(A,B)\), or with one of the characterizations \ref{mly_up_1} or \ref{mly_up_2} above.
\end{remark}
Now, our purpose in the present section is to build a comparison functor
\[\vxy{
		\bbG : \Mly(\Cat) \ar[r]& \TwoTDX
	}\]
thus specializing the construction of \(\Mly(\clK)\) to the case	\(\clK=\Cat\), and then to the case \(\clK=\Set\) to obtain a functor
\[\label{baumiao}\vxy{
		G : \Mly \ar[r]& \TDX
	}\]
from the bicategory \(\Mly = \Mly(\Set)\), to the bicategory of 1-transducers of \autoref{bic_of_1_tnd}.
First of all, we define \(\bbG\) to be the identity on objects; then, we observe that any Mealy automaton \((\clX,d,s) : \clA\to\clB\) between categories induces a functor
\[\vxy{D : \clA^*\times\clX \ar[r] & \clX \times\clB^*}\]
in the following way:
\begin{itemize}
	\item the functor \(d : \clA^*\times\clX \to\clX\) clearly has an extension to \(d^* : \clA^*\times\clX \to\clX\), given the universal	property of the \(\clA^*\);
	\item the functor \(s : \clX\times\clB^* \to\clB\) has an extension to \(s^* : \clA^*\times\clX \to\clB^*\), defined `inductively' (\cf \cite[Lemma 2.13]{EPTCS397.1}) as
	      \[
		      \begin{cases}
			      s^\flat(e, \emptyList) & = \emptyList                          \\
			      s^\flat(e, a \kons as) & = s(e, a) \kons s^\flat(d(e, a), as).
		      \end{cases}
	      \]
\end{itemize}
Now, let \(D:=\langle d^*,s^\flat\rangle\), and note that the representable profunctor \(\hom(1,D) : \clA\times\clX\tdto\clX\times\clB\) defined by \(D\) and given by
\[\big((x',\unb),(x,\una)\big)\mapsto\clX(x',d(\una,x))\times\clB^*(b,s^\flat(\una,x))\]
has precisely the correct type for being a \twoDash transducer \((\clX,\hom(1,D)) : \clA\tdto\clB\). This defined \(\bbG\) on 1-cells, so that it is a (pseudo)functor: it preserves identities, as the Mealy automaton
\[\vxy{
		\clA^*\times\clI\ar[r]^-\cong & \clA^* \ar@{=}[r] & \clA^* \ar[r]^-\cong & \clI\times\clA^*
	}\]
is sent to the transducer defined as
\begin{align*}
	\big((\bullet,\una'),(\una,\bullet)\big) & \mapsto \clI(\bullet, d^*(\una,\bullet))\times\clA^*(\una',s^\flat(\una,\bullet)) \\
	                                         & \cong I\times\clA^*(\una',\una)                                                   \\
	                                         & \cong \clA^*(\una',\una)
\end{align*}
It preserves compositions, as given transducers
\[\vxy[@C=1.5cm]{\clA \ar[r]^-{\bbG(\clX,d_1,s_1)} & \clB \ar[r]^-{\bbG(\clY,d_2,s_2)} & \clC}\]
their composition \emph{qua} transducers is given by
\begin{multline}
	\big((\una,(x,y)),((x',y'),\unc)\big) \longmapsto \\ \int^{\unb} \clX\big(x', d_1^*(\una,x)\big)\times\clB^*\big(b,s_1^\flat(\una,x)\big)\times\clY\big(y',d_2^*(\unb,y)\big)\times \clC^*\big(\unc,s_2^\flat(\unb,y)\big)\end{multline}
which can be reduced by a straightforward calculation to
\begin{align*}
	\kern1mm       & \bbG(\clY,d_2,s_2)\tranComp\bbG(\clX,d_1,s_1)                                                                                                                        \\
	\cong \kern1mm & \clX\big(x', d_1^*(\una,x)\big)\times\int^{\unb} \clB^*\big(b,s_1^\flat(\una,x)\big)\times\clY\big(y',d_2^*(\unb,y)\big)\times\clC^*\big(\unc,s_2^\flat(\unb,y)\big) \\
	\cong \kern1mm & \clX\big(x', d_1^*(\una,x)\big)\times \clY\big(y',d_2^*(s_1^\flat(\una,x),y)\big)\times\clC^*\big(\unc,s_2^\flat(s_1^\flat(\una,x),y)\big)                           \\
	\cong \kern1mm & (\clX\times\clY)\big((x',y'),(d_1^*(\una,x),d_2^*(s_1^\flat(\una,x),y))\big)\times \clC^*\big(\unc,s_2^\flat(\unb,y)\big)                                            \\
	=\kern1mm      & \bbG(\clX\times\clY,d_2\diamond d_1,s_2\diamond s_1)\big((\una,(x,y)),((x',y'),\unc)\big).
\end{align*}
On 2-cells, every functor \(F : \clX\to\clY\) such that the squares
\[\vxy{
	\clX\ar[d]_-F & \clA\times\clX \ar[r]^-{d_1}\ar[l]_-{s_1}\ar[d]^{\clA\times F}& \clB\ar@{=}[d]\\
	\clY & \clA\times\clY \ar[r]^-{d_2}\ar[l]_-{s_2}& \clB
	}\]
both commute is also that
\[\begin{cases}
		s_2^\flat(\una,Fx) = s_1^\flat(\una,x), \\
		F(d_1^*(\una,x)) = d_2^*(\una,Fx).
	\end{cases}\]
So, there is an induced 2-cell between the associated transducers, mediated by the functorial action of \(F\) and the embedding \(D\mapsto\hom(1,D)\): as a consequence there is a pseudo\-com\-mu\-ta\-ti\-ve square of profunctors
\[\vxy{
	\clA^*\times\clX \ar[r]|-@{|}\ar[d]|-@{|}& \clX\times\clB^*\ar[d]|-@{|}\\
	\clA^*\times\clY \ar[r]|-@{|}& \clY\times\clB^*
	}\]
whence a transformation with components induced by the action of \(F\),
\begin{align*}
	((x',\unb),(\una,x)) & \mapsto \bbG(D)((x',\unb),(\una,x))                            \\
	                     & =\clX(x',d_1(\una,x))\times\clB^*(\unb,s_1^\flat(\una,x))      \\
	                     & \to \clY(Fx',Fd_1(\una,x))\times\clB^*(\unb,s_1^\flat(\una,x)) \\
	                     & =\clY(Fx',d_2(\una,Fx))\times\clB^*(\unb,s_1^\flat(\una,Fx))   \\
	                     & =\bbG(D')((Fx',\unb),(\una,Fx))
\end{align*}
and filling the triangle
\[\vxy[@R=5mm]{
	\clA\times \opid\clX \drtwocell<\omit>{<3>\kern.5em\theta_F}
	\ar[dd]_-{\clA\times \opid F}\ar[dr]^-{\bbG(D)} & \\
	& \freerig{\clB} \\
	\clA\times \opid\clY. \ar[ur]_-{\bbG(D')}&
	}\]
Everything that has been done for \(\Set\)-enriched categories can be specialized to discrete ones, regarded as enriched over Booleans; this yields the 2-functor \(G\) of \eqref{baumiao}.
\begin{remark}\label{equippy}
	Summing up the results of this subsection, we have constructed a diagram
	\[\vxy{
			\Cat\ar[dr]_{\Pi}\ar[rr]^-J && \Mly(\Cat) \ar[dl]^-I \\
			& \TwoTDX
		}\]
	commutative up to iso, where the composite functor \(\Pi=IJ\) is a proarrow equipment, yet \(I\) and \(J\) separately are not proarrow equipments, since
	\begin{itemize}
		\item not every 1-cell in the image of \(J\) has a conjoint; in fact, only the companions of isomorphisms of categories admit a conjoint in \(\Mly(\Cat)\);
		\item not every 1-cell in the image of \(I\) has a conjoint: in order to have a conjoint it is necessary to have terminal category of states. But then, every 1-cell in \(\Mly(\Cat)\), in the image of \(J\), has a conjoint when mapped through \(I\).
	\end{itemize}
\end{remark}
\subsection{\(\TwoTDX\) and Guitart's \(\MAC\)}
The relation of \(\TwoTDX\) with another notable bicategory of `machines' is more subtle; let's start recalling the definition of the bicategory \(\MAC\) of `machines' that René Guitart gave in \cite[§2]{guitart1974remarques}. We are basically repeating an outline of Guitart's definition that appears in \cite{EPTCS397.1}.
\begin{definition}
	\(\MAC\) is the bicategory having
	\begin{itemize}
		\item objects the small categories \(\clA,\clB,\dots\);
		\item as hom-categories \(\Mac(\clA,\clB)\) the full subcategory of \(\Span(\Cat)(\clA,\clB)\) spanned by diagrams
		      \[\label{macinno}\vxy{
				      &\clE\ar[dr]^g\ar[dl]_q &\\
				      \clA&&\clB
			      }\]
		      where \(q\) is a discrete opfibration.
	\end{itemize}
\end{definition}
We informally refer to 1-cells of \(\MAC\) as `machines'.

Recall from \cite{Jay1988} that a \emph{local adjunction} between bicategories \(\clX,\clY\) consists of lax functors \(L:\clX \rightleftarrows\clY:R\) equipped with a family of adjunctions
\[\adja{\lambda_{XY}}{\clY(LX,Y)}{\clX(X,RY)}{\rho_{XY}}\]
By evident extension of this terminology, a \emph{local reflection} is a local adjunction where each \(\rho_{XY}\) is fully faithful. The following adapts a well-known fact about profunctors represented as spans.
\begin{lemma}
	There is a local reflection
	\[\adja{(\blank)^\varphi}{\MAC}{\Prof}{j}\]
	where the bicategory on the right hand side is profunctors, regarded as two-sided discrete fibrations, inside \(\Span(\Cat)\).
\end{lemma}
The purpose of the present subsection is to prove the following result.
\begin{theorem}
	There is a local reflection
	\[\adja{\Rbag\blank}{\TwoTDX}{\Prof}{\bsu\blank}\]
	induced by the span representation of profunctors.
\end{theorem}
\begin{proof}
	For every \(\clA,\clB : \Cat\) we have to define a reflection
	\[\adja{(\Rbag\blank)_{\clA\clB}}{\TwoTDX(\clA,\clB)}{\Prof(\clA,\clB^*)}{(\bsu\blank)_{\clA\clB}.}\]
	We define
	\begin{itemize}
		\item \(\Rbag\big((\clQ,t) : \clA\tdto\clB\big)\) as the machine obtained from  a \twoDash transducer \((\clQ,t)\) as follows:
		      \begin{itemize}
			      \item first, take the colimit \(\int t = \colim_{\opid\clQ} \big(t(\blank,q,q)\big) : \clA\times(\clB^*)^\op\to\Set\) (not the coend, see below);
			      \item then, turn \(\int t\) into a two-sided discrete fibration
			            \[\vxy{
					            &\clE(\int t)\ar[dr]^{p_t}\ar[dl]_{q_t}& \\
					            \clA && \clB^*,
				            }\label{eq:discrete_fibration_from_colimit}\]
			            using the well-known equivalence of categories.
		      \end{itemize}
		\item \(\bsu(\clE;q,p)\) as the \twoDash transducer obtained from a two-sided discrete fibration regarding it as a profunctor \(\clA\times(\clB^*)^\op\to\Set\), and thus as a transducer
		      \[\vxy{
				      \clA\times\opid\clI(\clB^*)^\op\ar[r] & \Set.
			      }\]
	\end{itemize}
	Now start with a transducer \((\clQ,t)\) and apply \(\bsu\) to \(\Rbag(\clQ,t)=(\int t,p_t,q_t)\); then there is a transformation with components
	\[\vxy{\eta : (\clQ,t) \ar@{=>}[r] & \bsu\big(\Rbag(\clQ,t)\big)}\]
	in turn having components
	\[\vxy{
			\eta^2 : t(\una,q,q')(\unb) \ar[r] & \colim_{qq'}t(\una,q,q')(\unb)
		}\]
	induced by the initial cocone defining \(\int t\) (defining it as the coend would have given just a cowedge here, not a natural transformation), and giving rise to a 2-cell in \(\TwoTDX\) of shape
	\[\vxy[@R=5mm]{
			\clA\times\opid\clQ\ar[dr] \ar[dd]\drtwocell<\omit>{<3>\eta^2}& \\
			&\freerig\clB\\
			\clA\times\opid\clI\ar[ur] &
		}\]
	yielding the desired unit of the reflection; the counit is indeed an isomorphism, as the composition \(\Rbag\big(\bsu(\clE;q,g)\big)\) consists of starting with a two-sided discrete fibration, regard it as a profunctor, and revert it back to a two-sided discrete fibration.

	In a fibrational perspective, we are identifying \(\Prof\) with the fiber at \(\clI\) of a functor \(\TwoTDX(\clA,\clB)\to\Cat:(\clQ,\blank)\mapsto\clQ\); this fiber is reflective.
\end{proof}
\section{Future plans}\label{future}
We end the paper sketching two possible directions for future investigation. Admittedly, some parts of it remain to this day a bit conjectural, but we believe there is some deep observation lurking behind them, that deserves mention, and can further convince the reader that categorification allowed to unveil an interesting pattern.
\subsection{A more general double category of transducers}\label{sec:more-general}
Along the discussion so far, we studied transducers between free objects, in an attempt to retain some intuition about the classical theory of processes as representations of a free monoid on an alphabet.

There have been attempts (cf. the former Guitart approach in \cite{guitart1974remarques}) where such a restraint was abandoned, considering processes between non-free monoids. We would like to study where this idea goes, and define a \emph{double category} of transducers between general monoidal categories. Part of the motivation for doing so is the same as Guitart's: most of the results outlined so far extend without much effort, \ie they are completely independent on the freeness assumption on the domain and codomain of a \twoDash transducer, but the embedding result of \(\Mly(\Cat)\) relies on a more refined compatibility between the output map \(S\) of a span
\[\vxy{
		\clX & \clM\times\clX \ar[r]^-S \ar[l]_-D & \clN,
	}\]
namely the property that there is a natural isomorphism in \(\clN\),
\[\vxy{\fugue SMX{M'} \ar[r]^-{\phi_{XMM'}} & S(M\otimes M',X)}\]
subject to suitable compatibility conditions.

\medskip
Another reason to indulge in such a maximal level of generality is that this refined notion, albeit burdening the analysis with a lot of new coherence problems, is neither new nor artificial. It first appeared in \cite{guitart1974remarques} without particular explanation besides its utility, and was further developed in \cite[§2.1]{EPTCS397.1}, where the condition was termed the \emph{fugal property} for \( S \).\footnote{Guitart does not give a name to this property; in \cite{EPTCS397.1} we call it \emph{fugality} for \(s\) (\cf [\emph{ibi}] for a thorough description of the matter); it can be motivated with the fact that this requirement on \(s\) is necessary and sufficient for \(s\) to induce a functor \(\clE[d^+]\to B\), and thus a machine, where \(\clE[d^+]\) is the domain of the discrete opfibration over \(A^*\) naturally associated to the representation \(d^+\).} The same notion was further studied (and in large part mechanized in a proof-assistant) in \cite[Definition 2.18]{loregian2025monadslimitsbicategoriescircuits}, finding that in a double category modeled on \(\Mly(\Set)\), a monad map between monads \( M \) and \( N \) is exactly equivalent to a third Mealy automaton between the state spaces of \( M \) and \( N \), whose output function satisfies the fugal property. What may initially appear to be an \emph{ad hoc} modification of standard definitions, therefore, arises instead quite naturally from first principles.

\medskip
The study of a categorified notion of fugality is part of a larger work in progress, taking its initial steps from the observation that
\begin{itemize}
	\item there is a larger bicategory \(\VMealy\) where usual Mealy automata embed, outlined by Paré in his \cite{Par2010}; such a bicategory seems the most suitable candidate for a `bicategory of general automata' as it is equipped with a clear-cut universal property and extends (in the sense that there are embedding from the latter into the former) virtually all examples, such as \cite{Katis1997,guitart1974remarques} of bicategories of automata;
	\item the pseudo double category having \(\VMealy\) as bicategory of loose arrows contains (via a local coreflection, \cf \cite{Jay1988}, the notion has been exploited in \cite[Lemma 2.13]{EPTCS397.1}) all the double categories outlined here, the double category of `circuits' outlined in \cite{loregian2025monadslimitsbicategoriescircuits}, as well as the double categorical analogue of Guitart's \(\MAC\).
\end{itemize}
\begin{definition}[The double category of transducers between monoidal categories]
	\label{def:double-category-transducers}
	The \emph{double category of monoidal transducers} \(\bbM\DTDX\) is defined as follows:
	\begin{itemize}
		\item \textbf{0-cells} are monoidal categories \(\clM,\clN,\dots\);
		\item a tight cell \(F : \clM\to\clN\) is a strong monoidal functor between the underlying categories of \(\clM\) and \(\clN\);
		\item a loose cell \((\clQ,t) : \clM \tdto \clN\) are pairs where \(\clQ\) is a category and \(t\) a functor of type
		      \[\vxy{
				      t : \clM \times \opid{\clQ} \times \clN^\op \ar[r] & \clV
			      }\]
		\item a cell  \((u,\alpha)\) with frame
		      \[\vxy{
			      \clM \ar[d]_F \ar[r]|-@{*}^-{(\clQ,s)}& \clN\ar[d]^G \\
			      \clM' \ar[r]|-@{*}_-{(\clP,t)}& \clN'
			      }\]
		      consists of a pair where \(u : \clQ\to\clP\) is a functor and \(\alpha\) is a natural transformation with components
		      \[\vxy{
				      \alpha : s(a,q,q')(b) \ar[r] & t(Fa,uq,uq')(Gb),
			      }\]
		      \ie filling the diagram
		      \[\vxy{
				      \clM \times \opid{\clQ} \times \clN^\op \ar[dr]^s\ar[dd]_{F\times\opid u\times G}\drtwocell<\omit>{<4>\alpha}&\\
				      & \clV\\
				      \clM' \times \opid{\clP} \times (\clN')^\op \ar[ur]_t
			      }\]
	\end{itemize}
\end{definition}
\subsection{More on the linear algebra, and representation theory, of profunctors}
The mention of `structure constants' in \autoref{remark_structure_constants_of_a_2_transducer} gives some leeway to speculate on the theory of transducers from a representation\hyp{}theoretic point of view; in particular, what is their geometric meaning when they are interpreted as categorified matrices?

If in \autoref{1_transduc} the base semiring \(\bbK\) is taken to be a field (better, algebraically closed of characteristic zero), a 1-transducer as in \eqref{generic_1_tdx} consists precisely of a representation of the free monoid \(A^*\) into the Lie algebra of \(n\times n\) matrices valued in Laurent series over \(\bbK\).

One classical context where such Laurent-series-valued matrices arise is the study of linear ODEs with regular singularities, and especially meromorphic connections on vector bundles \cite{Wasow1965, Levelt1971,Deligne1970, Malgrange1991,Boalch2001, Sabbah2007}, a well\hyp{}established area of algebro-geometric approaches to the analysis of ODEs.

In short: when given a differential equation of type
\[Y'(z) = A(z)Y(z)\]
where \(A(z)\in\fkg\fkl_n(\laurent{\bbC}{z})\)
one attempts to classify its solutions looking at the structure of \(A\), thought as a matrix of germs of meromorphic functions. If, say, \(A(z)\) has `order 1', \ie it admits an expansion \(\frac{A_{-1}}z + A_0 + A_1 z + \dots\) then the constant matrix \(A_{-1}\) describes the \emph{monodromy} of the differential equation. For higher orders, the equation has \emph{irregular} singularities the classification of which requires refined (highly nontrivial) tools.

\medskip
Motivated by this perspective, we wonder what is the theory of transducers from a `Lie-theoretic perspective'; for example, the structure constants \((t_a \mid a \in A)\) of \autoref{remark_structure_constants_of_a_2_transducer} uniquely determine \(t\) due to the freeness property of \(A^*\); one can study the subset of \(A\) yielding pairwise-commuting \(t_a\)'s (under composition of matrices/profunctors); the elements of such a set are simultaneously diagonalizable, a property which turns out to be a cornerstone integrable systems theory \cite{Kostant1979, BabelonBernardTalon2003}.

Or, again over the field of Laurent series with complex coefficients, any matrix-valued function \(\Phi(z)\) in \(\fkg\fkl_n(K)\) gives rise to a \emph{spectral curve} \cite{Beauville1995, DonagiMarkman1996}, defined by means of the characteristic equation \(\det(\lambda I - \Phi(z)) = 0\). This allows the study of the parametric family of operators given by \(\Phi\) in terms of line bundles over the spectral curve.

To each 1-transducer \(t\) one can associate the function \(C_t : A\to K[\lambda]\) so that \(C_t(a) = \det(\lambda I - t_a(z))\), and study the locus of \(z\) for which \(C_t(a)\) vanishes, be it as a mere subset or as a subspace.\footnote{For example if \(A\) is a topological space, the list construction yields a topological monoid \(A^*\) \cite{Goubault2019}, of which the \emph{commutant}
	\[\mathrm{Comm}(A) := \left\{ a \in A \mid [t_a, t_b] = 0 \text{ for all } b \in A \right\}.\]
	is a closed subspace.}

\medskip
We leave the suggestive open question of whether there is something to gain from this formal analogy, and what are the interesting `Lie theoretic' invariants to attach a 1-transducer, open for future investigation.

\hrulefill
\subsubsection*{Acknowledgments}
This work was supported by the ESF funded Estonian IT Academy research measure (project 2014-2020.4.05.19-0001).

I am grateful to Nathanael Arkor for organizing the weekly seminar on double category theory at TalTech---without which I would not have gained sufficient command of the theory to attempt this work---and for his thorough proofreading of the manuscript. Any remaining mistakes are solely my responsibility.

My interest in this problem was rekindled in April 2025 during the Dagstuhl Seminar 25141, \emph{Categories for Automata and Language Theory}. A week of cocontinuous discussions with Paul-André Melliès, Daniela Petri\c{s}an, Victor Iwaniack, Zeinab Galal, Noam Zeilberger, and many others shaped several of the ideas presented here. I warmly thank them all. Stefano Kasangian provided useful feedback and a pointer to \cite{Kasangian2010}.
\putbib{refs}{amsalpha}

\newcommand{\etalchar}[1]{$^{#1}$}
\providecommand{\bysame}{\leavevmode\hbox to3em{\hrulefill}\thinspace}
\providecommand{\MR}{\relax\ifhmode\unskip\space\fi MR }
\providecommand{\MRhref}[2]{%
  \href{http://www.ams.org/mathscinet-getitem?mr=#1}{#2}
}
\providecommand{\href}[2]{#2}
\begin{thebibliography}{BLLL23b}

\bibitem[ABM14]{Agore_2014}
A.~L. Agore, C.-G. Bontea, and G.~Militaru, \emph{Classifying bicrossed
  products of {H}opf algebras}, Algebras and Representation Theory (2014).

\bibitem[ACIM09]{Agore2009forfinite}
A.~L. Agore, A.~Chirvăsitu, B.~Ion, and G.~Militaru, \emph{Bicrossed products
  for finite groups}, Algebras and Representation Theory \textbf{12} (2009),
  no.~2–5, 481–488.

\bibitem[AK88]{Albert1988}
M.H. Albert and G.M. Kelly, \emph{The closure of a class of colimits}, Journal
  of Pure and Applied Algebra \textbf{51} (1988), no.~1–2, 1–17.

\bibitem[AM10]{Agore_2010}
A.~L. Agore and G.~Militaru, \emph{Extending structures {I}: Unifying crossed
  and bicrossed products}, 2010.

\bibitem[AM12]{Agore_2009}
\bysame, \emph{Schreier type theorems for bicrossed products}, Open Mathematics
  \textbf{10} (2012), no.~2, 722--739 (eng).

\bibitem[AMV11]{adamek2011coalgebraic}
J.~Ad{\'a}mek, S.~Milius, and J.~Velebil, \emph{Coalgebraic analysis of
  probabilistic systems}, Theoretical Computer Science \textbf{412} (2011),
  no.~38, 4989--5009.

\bibitem[BBT03]{BabelonBernardTalon2003}
O.~Babelon, D.~Bernard, and M.~Talon, \emph{Introduction to classical
  integrable systems}, Cambridge University Press, Cambridge, 2003.

\bibitem[Bea95]{Beauville1995}
A.~Beauville, \emph{Jacobians of spectral curves and completely integrable
  systems}, Integrable Systems and Quantum Groups \textbf{1620} (1995), 25--31.

\bibitem[Ber13]{Berstel1979}
Jean Berstel, \emph{Transductions and context-free languages}, Vieweg+Teubner
  Verlag, 2013.

\bibitem[BFL{\etalchar{+}}23]{boccali2023semibicategory}
G.~Boccali, B.~Femić, A.~Laretto, F.~Loregian, and S.~Luneia, \emph{The
  semibicategory of {M}oore automata}, 2023.

\bibitem[BK81]{Betti1981}
R.~Betti and S.~Kasangian, \emph{A quasi-universal realization of automata.},
  Universit{\`a} degli Studi di Trieste. Dipartimento di Scienze Matematiche
  (1981), 41--48.

\bibitem[BK82]{Betti1982}
\bysame, \emph{Una proprietà del comportamento degli automi completi.},
  Universit{\`a} degli Studi di Trieste. Dipartimento di Scienze Matematiche
  (1982), 17--26.

\bibitem[BLL98]{bergeron1998combinatorial}
F.~Bergeron, G.~Labelle, and P.~Leroux, \emph{Combinatorial species and
  tree-like structures}, Encyclopedia of Mathematics and its Applications,
  no.~67, Cambridge University Press, Cambridge, 1998.

\bibitem[BLLL23a]{EPTCS397.1}
G.~Boccali, A.~Laretto, F.~Loregian, and S.~Luneia, \emph{Bicategories of
  automata, automata in bicategories}, Proceedings of the Sixth International
  Conference on Applied Category Theory 2023 (University of Maryland) (Sam
  Staton and Christina Vasilakopoulou, eds.), Electronic Proceedings in
  Theoretical Computer Science, vol. 397, Open Publishing Association, 2023,
  pp.~1--19.

\bibitem[BLLL23b]{boccali_et_al:LIPIcs.CALCO.2023.20}
\bysame, \emph{Completeness for categories of generalized automata}, 10th
  Conference on Algebra and Coalgebra in Computer Science (CALCO 2023)
  (Dagstuhl, Germany) (Paolo Baldan and Valeria de~Paiva, eds.), Leibniz
  International Proceedings in Informatics (LIPIcs), vol. 270, Schloss Dagstuhl
  -- Leibniz-Zentrum f{\"u}r Informatik, 2023, pp.~20:1--20:14.

\bibitem[BMT23]{baez2023schurfunctorscategorifiedplethysm}
J.C. Baez, J.~Moeller, and T.~Trimble, \emph{Schur functors and categorified
  plethysm}, 2023.

\bibitem[BMT24]{baez20242rigextensionssplittingprinciple}
John~C. Baez, Joe Moeller, and Todd Trimble, \emph{2-rig extensions and the
  splitting principle}, 2024.

\bibitem[Boa01]{Boalch2001}
P.~Boalch, \emph{Symplectic manifolds and isomonodromic deformations}, Advances
  in Mathematics \textbf{163} (2001), no.~2, 137--205.

\bibitem[Bor94]{Bor2}
F.~Borceux, \emph{Handbook of categorical algebra 2. (categories and
  structures)}, Encyclopedia of Mathematics and its Applications, vol.~51,
  Cambridge University Press, Cambridge, 1994.

\bibitem[BS00]{Benabou2000}
J.~Bénabou and T.~Streicher, \emph{Distributors at work}, Lecture notes
  written by Thomas Streicher, 2000.

\bibitem[Cho04]{Choffrut2004}
C.~Choffrut, \emph{Rational relations as rational series}, p.~29–34,
  Springer, Berlin, 2004.

\bibitem[Del70]{Deligne1970}
P.~Deligne, \emph{Équations différentielles à points singuliers réguliers},
  Lecture Notes in Mathematics, vol. 163, Springer, Berlin, 1970.

\bibitem[DL07]{DAY2007651}
B.J. Day and S.~Lack, \emph{Limits of small functors}, Journal of Pure and
  Applied Algebra \textbf{210} (2007), no.~3, 651--663.

\bibitem[DM96]{DonagiMarkman1996}
R.~Donagi and E.~Markman, \emph{Spectral covers, algebraically completely
  integrable, hamiltonian systems, and moduli of bundles}, Integrable Systems
  and Quantum Groups \textbf{1620} (1996), 1--119.

\bibitem[DMS19]{Dorsch2018GradedMA}
U.~Dorsch, S.~Milius, and L.~Schr\"{o}der, \emph{{G}raded monads and graded
  logics for the linear time - branching time spectrum}, 30th International
  Conference on Concurrency Theory (CONCUR 2019) (Dagstuhl, Germany), LIPIcs,
  vol. 140, 2019, pp.~36:1--36:16.

\bibitem[Eis02]{eisner2002parameter}
J.~Eisner, \emph{Parameter estimation for probabilistic finite-state
  transducers}, Proceedings of the 40th Annual Meeting of the Association for
  Computational Linguistics, 2002, pp.~1--8.

\bibitem[Elg21]{elgueta2020groupoid}
J.~Elgueta, \emph{The groupoid of finite sets is biinitial in the 2-category of
  rig categories}, Journal of Pure and Applied Algebra \textbf{225} (2021),
  no.~11, 106738.

\bibitem[FKM16]{Fujii2016}
S.~Fujii, S.~Katsumata, and P.-A. Melli{\`e}s, \emph{Towards a formal theory of
  graded monads}, FoSSaCS. Foundations of Software Science and Computation
  Structures (B.~Jacobs and C.~L{\"o}ding, eds.), 2016, pp.~513--530.

\bibitem[Fuj19]{fujii20192categorical}
S.~Fujii, \emph{A 2-categorical study of graded and indexed monads}, 2019,
  \arXivPreprint{1904.08083}.

\bibitem[GJ17]{Gambino2017}
N.~Gambino and A.~Joyal, \emph{On operads, bimodules and analytic functors},
  Memoirs of the American Mathematical Society \textbf{249} (2017), no.~1184,
  0–0.

\bibitem[GKO{\etalchar{+}}16]{10.1145/3022670.2951939}
M.~Gaboardi, S.~Katsumata, D.~Orchard, F.~Breuvart, and T.~Uustalu,
  \emph{Combining effects and coeffects via grading}, SIGPLAN Not. \textbf{51}
  (2016), no.~9, 476–489.

\bibitem[GKOS21]{Gaboardi2021}
M.~Gaboardi, S.~Katsumata, D.~Orchard, and T.~Sato, \emph{Graded {H}oare logic
  and its categorical semantics}, p.~234–263, Springer International
  Publishing, Luxembourg, 2021.

\bibitem[GM19]{Goubault2019}
E.~Goubault and S.~Mimram, \emph{Presentations of higher-dimensional automata},
  Logical Methods in Computer Science \textbf{15} (2019), no.~1, 1--38.

\bibitem[GP99]{grandispare1999limits}
M.~Grandis and R.~Paré, \emph{Limits in double categories}, Cahiers de
  Topologie et G\'eom\'etrie Diff\'erentielle Cat\'egoriques \textbf{40}
  (1999), no.~3, 162--220 (en). \MR{1716779}

\bibitem[GPS95]{coherence-tricat}
R.~Gordon, A.J. Power, and R.~Street, \emph{Coherence for tricategories}, Mem.
  Amer. Math. Soc. \textbf{117} (1995), no.~558, vi+81.

\bibitem[Gri07]{grillet}
P.~Grillet, \emph{Abstract algebra}, 2007, Graduate Texts in Mathematics,
  Springer New York.

\bibitem[GS16]{Garner2016}
R.~Garner and M.~Shulman, \emph{Enriched categories as a free cocompletion},
  Advances in Mathematics \textbf{289} (2016), 1–94.

\bibitem[Gui74]{guitart1974remarques}
R.~Guitart, \emph{Remarques sur les machines et les structures}, Cahiers de
  topologie et g\'eom\'etrie diff\'erentielle \textbf{15} (1974), no.~2,
  113--144 (fr). \MR{384891}

\bibitem[Gui78]{guitart1978bimodules}
\bysame, \emph{Des machines aux bimodules}, Univ. Paris 7, apr 1978.

\bibitem[Gui80]{Guitart1980}
\bysame, \emph{Tenseurs et machines}, Cahiers de topologie et g\'eom\'etrie
  diff\'erentielle \textbf{21} (1980), no.~1, 5--62 (fr). \MR{569117}

\bibitem[Han10]{silva2010subsequential}
H.H. Hansen, \emph{Subsequential transducers: a coalgebraic perspective},
  Information and Computation \textbf{208} (2010), no.~12, 1368–1397.

\bibitem[Hor24]{hora2024topoiautomataitopoi}
R.~Hora, \emph{Topoi of automata {I}: Four topoi of automata and regular
  languages}, 2024.

\bibitem[HS19]{hansen2019constructingsymmetricmonoidalbicategories}
L.W. Hansen and M.~Shulman, \emph{Constructing symmetric monoidal bicategories
  functorially}, 2019.

\bibitem[Hum72]{humphreys1972introduction}
J.E. Humphreys, \emph{Introduction to {L}ie algebras and representation
  theory}, Graduate Texts in Mathematics, vol.~9, Springer, Berlin, 1972.

\bibitem[IR19]{Ibort2019}
A.~Ibort and M.A. Rodríguez, \emph{An introduction to groups, groupoids and
  their representations}, October 2019.

\bibitem[Iwa24]{Iwaniack2024}
V.~Iwaniack, \emph{Automata in {$W$}-toposes, and general {M}yhill-{N}erode
  theorems}, p.~93–113, Springer Nature, Switzerland, 2024.

\bibitem[Jay88]{Jay1988}
C.B. Jay, \emph{Local adjunctions}, Journal of Pure and Applied Algebra
  \textbf{53} (1988), no.~3, 227–238.

\bibitem[JB88]{book91519310}
C.~Reutenauer J.~Berstel, \emph{Rational series and their languages}, 2008
  electronic ed., EATCS Monographs on Theoretical Computer Science 12,
  Springer, Berlin, 1988.

\bibitem[Kar18]{karvonen2019waydagger}
M.~Karvonen, \emph{The way of the dagger}, Phd thesis, University of Edinburgh,
  School of Informatics, Laboratory for Foundations of Computer Science,
  Edinburgh, UK, 2018.

\bibitem[Kel05]{kelly}
G.M. Kelly, \emph{Basic concepts of enriched category theory}, Repr. Theory
  Appl. Categ. \textbf{64} (2005), no.~10, vi+137. \MR{2177301}

\bibitem[KKR83]{kasangian1983cofibrations}
S.~Kasangian, G.M. Kelly, and F.~Rossi, \emph{Cofibrations and the realization
  of non-deterministic automata}, Cahiers de topologie et g{\'e}om{\'e}trie
  diff{\'e}rentielle cat{\'e}goriques \textbf{24} (1983), no.~1, 23--46.

\bibitem[KL80]{Kelly1980}
G.M. Kelly and L.~Laplaza, \emph{Coherence for compact closed categories},
  Journal of Pure and Applied Algebra \textbf{19} (1980), 193--213.

\bibitem[KL10]{Kasangian2010}
S.~Kasangian and A.~Labella, \emph{Conduché property and {$Tree$}-based
  categories}, Journal of Pure and Applied Algebra \textbf{214} (2010), no.~3,
  221–235.

\bibitem[Kos79]{Kostant1979}
B.~Kostant, \emph{The solution to a generalized {T}oda lattice and
  representation theory}, Advances in Mathematics \textbf{34} (1979), no.~3,
  195--338.

\bibitem[KR90]{Kasangian1990}
S.~Kasangian and R.~Rosebrugh, \emph{Glueing enriched modules and composition
  of automata}, Cahiers de topologie et g{\'e}om{\'e}trie diff{\'e}rentielle
  cat{\'e}goriques \textbf{31} (1990), no.~4, 283--290.

\bibitem[KS85]{Kuich1985yp}
W.~Kuich and A.~Salomaa, \emph{Semirings, automata, languages}, 1986 ed.,
  Monographs in Theoretical Computer Science. An EATCS Series, Springer,
  Berlin, Germany, December 1985 (en).

\bibitem[KSW97]{Katis1997}
P.~Katis, N.~Sabadini, and R.F.C. Walters, \emph{Bicategories of processes},
  Journal of Pure and Applied Algebra \textbf{115} (1997), no.~2, 141--178.

\bibitem[KSW02]{Katis2002}
P.~Katis, N.~Sabadini, and R.F.C. Walters, \emph{Feedback, trace and
  fixed-point semantics}, RAIRO - Theoretical Informatics and Applications -
  Informatique Th\'eorique et Applications \textbf{36} (2002), no.~2, 181--194
  (en). \MR{1948768}

\bibitem[KSW10]{Katis2010}
\bysame, \emph{Feedback, trace and fixed-point semantics}, RAIRO - Theoretical
  Informatics and Applications \textbf{36} (2010), no.~2, 181--194 (eng).

\bibitem[Lap72]{laplaza1972coherence}
M.L. Laplaza, \emph{Coherence for distributivity}, Coherence in categories,
  Springer, Berlin, 1972, pp.~29--65.

\bibitem[Lev71]{Levelt1971}
A.~H.~M. Levelt, \emph{Jordan decomposition for a class of singular
  differential operators}, Arkiv för Matematik \textbf{13} (1971), no.~1,
  1--27.

\bibitem[LGR{\etalchar{+}}21]{openTransitionSystems21}
E.~Di Lavore, A.~Gianola, M.~Rom{\'{a}}n, N.~Sabadini, and P.~Soboci\'nski,
  \emph{A canonical algebra of open transition systems}, Formal Aspects of
  Component Software - 17th International Conference, {FACS} 2021 (Virtual
  Event) (Gwen Sala{\"{u}}n and Anton Wijs, eds.), Lecture Notes in Computer
  Science, vol. 13077, Springer, 2021, pp.~63--81.

\bibitem[Lor24]{loregian2024automata}
F.~Loregian, \emph{Automata and coalgebras in categories of species}, Lecture
  Notes in Computer Science (Switzerland), Springer Nature, 2024, p.~65–92.

\bibitem[Lor25]{loregian2025monadslimitsbicategoriescircuits}
\bysame, \emph{Monads and limits in bicategories of circuits}, 2025,
  \arXivPreprint{2501.01882}.

\bibitem[LT23]{Loregian2023}
F.~Loregian and T.~Trimble, \emph{Differential 2-rigs}, Electronic Proceedings
  in Theoretical Computer Science \textbf{380} (2023), 159–182.

\bibitem[LW25]{lucyshynwright2025vgradedcategoriesvwbigradedcategories}
R.B.B. Lucyshyn-Wright, \emph{{$V$}-graded categories and {$V$-$W$}-bigraded
  categories: Functor categories and bifunctors over non-symmetric bases},
  2025.

\bibitem[Mal91]{Malgrange1991}
B.~Malgrange, \emph{Équations différentielles à coefficients polynomiaux},
  Progress in Mathematics \textbf{96} (1991), 313--371.

\bibitem[Moh04]{Mohri2004}
M.~Mohri, \emph{Weighted finite-state transducer algorithms: An overview},
  Formal Languages and Applications (Carlos Martín-Vide, Victor Mitrana, and
  Gheorghe Paun, eds.), Studies in Fuzziness and Soft Computing, vol. 148,
  Springer, Berlin, 2004, pp.~551--564.

\bibitem[MPR01]{mohri2001weighted}
M.~Mohri, F.~Pereira, and M.~Riley, \emph{Weighted finite-state transducers in
  speech recognition}, Computer Speech \& Language \textbf{16} (2001), no.~1,
  69--88.

\bibitem[MPS15]{milius_et_al:LIPIcs:2015:5538}
S.~Milius, D.~Pattinson, and L.~Schr{\"o}der, \emph{{G}eneric trace semantics
  and graded monads}, 6th Conference on Algebra and Coalgebra in Computer
  Science ({CALCO}2015), LIPIcs 2015, 2015, pp.~253--269.

\bibitem[MU22a]{McDermott2022}
D.~McDermott and T.~Uustalu, \emph{Flexibly graded monads and graded algebras},
  p.~102–128, Springer International Publishing, Cham, 2022.

\bibitem[MU22b]{10.1007/978-3-031-16912-0_4}
\bysame, \emph{Flexibly graded monads and graded algebras}, Mathematics of
  Program Construction (Cham), Springer International Publishing, 2022,
  pp.~102--128.

\bibitem[MZ23a]{mellies:hal-04399404}
P.-A. Melli{\`e}s and N.~Zeilberger, \emph{{The categorical contours of the
  {Chomsky-Sch{\"u}tzenberger} representation theorem}}, December 2023.

\bibitem[MZ23b]{entics:10508}
P.-A. Melliès and N.~Zeilberger, \emph{Parsing as a lifting problem and the
  {Chomsky-Sch\"utzenberger} representation theorem}, Electronic Notes in
  Theoretical Informatics and Computer Science \textbf{1} (2023).

\bibitem[MZ25]{Mellis2025}
P.-A. Melliès and N.~Zeilberger, \emph{The categorical contours of the
  {Chomsky-Sch\"utzenberger} representation theorem}, Logical Methods in
  Computer Science \textbf{21} (2025).

\bibitem[OWE20]{Orchard2020}
D.~Orchard, P.~Wadler, and H.~Eades, \emph{Unifying graded and parameterised
  monads}, Electronic Proceedings in Theoretical Computer Science \textbf{317}
  (2020), 18--38.

\bibitem[Par10]{Par2010}
R.~Paré, \emph{Mealy morphisms of enriched categories}, Applied Categorical
  Structures \textbf{20} (2010), no.~3, 251–273.

\bibitem[Par21a]{Par2021}
\bysame, \emph{Morphisms of rings}, Joachim Lambek: The Interplay of
  Mathematics, Logic, and Linguistics (Claudia Casadio and Philip~J. Scott,
  eds.), Springer International Publishing, Cham, 2021, pp.~271--298.

\bibitem[Par21b]{pare2021three}
\bysame, \emph{Three easy pieces: Imaginary seminar talks in honour of {B}ob
  {R}osebrugh}, Theory and Applications of Categories \textbf{36} (2021),
  no.~6, 171--200.

\bibitem[Raj93]{rajan1993derivatives}
D.~S. Rajan, \emph{The adjoints to the derivative functor on species}, Journal
  of Combinatorial Theory, Series A \textbf{62} (1993), no.~1, 93--106.

\bibitem[RW88]{rosebrugh1988proarrows}
R.~Rosebrugh and R.J. Wood, \emph{Proarrows and cofibrations}, Journal of Pure
  and Applied Algebra \textbf{53} (1988), no.~3, 271--296.

\bibitem[Sab07]{Sabbah2007}
C.~Sabbah, \emph{Isomonodromic deformations and {F}robenius manifolds: An
  introduction}, Universitext, Springer, Berlin, 2007.

\bibitem[Ser92]{serre1992lie}
J.-P. Serre, \emph{Lie algebras and {L}ie groups}, Lecture Notes in
  Mathematics, vol. 1500, Springer, Berlin Heidelberg, 1992.

\bibitem[SS78]{Salomaa1978-sx}
A.~Salomaa and M.~Soittola, \emph{Automata-theoretic aspects of formal power
  series}, 1978 ed., Monographs in Computer Science, Springer, New York, NY,
  March 1978 (en).

\bibitem[SSVW17]{spipacchione}
P.~Schultz, D.~Spivak, C.~Vasilakopoulou, and R.~Wisnesky, \emph{Algebraic
  databases}, Theory and Applications of Categories \textbf{32} (2017), no.~16,
  547--619.

\bibitem[Wal89]{Walters1989}
R.F.C. Walters, \emph{A note on context-free languages}, Journal of Pure and
  Applied Algebra \textbf{62} (1989), no.~2, 199–203.

\bibitem[Was65]{Wasow1965}
W.~Wasow, \emph{Asymptotic expansions for ordinary differential equations},
  Pure and Applied Mathematics, Interscience Publishers, New York, 1965.

\bibitem[Woo76]{wood1976indicial}
R.J. Wood, \emph{Indicial methods for relative categories}, Ph.d. thesis,
  Dalhousie University, 1976.

\bibitem[Woo78]{Wood1978}
\bysame, \emph{{$\mathcal{V}$}-indexed categories}, p.~126–140, Springer,
  Berlin Heidelberg, 1978.

\bibitem[Woo82]{wood1982abstract}
\bysame, \emph{{A}bstract proarrows {I}}, Cahiers de topologie et géometrie
  différentielle categoriques \textbf{23} (1982), no.~3, 279--290.

\bibitem[Woo85]{wood1985proarrows}
\bysame, \emph{{P}roarrows {II}}, Cahiers de topologie et géométrie
  différentielle catégoriques \textbf{26} (1985), no.~2, 135--168.

\end{thebibliography}
\end{document}